\documentclass[a4paper,12pt,leqno]{amsart}
\usepackage{amsmath,amscd,amssymb,amsfonts}
\usepackage[matrix,arrow]{xy}
\usepackage{xcolor,hyperref}
\hypersetup{colorlinks=true,citecolor=black,linkcolor=black}
\newcommand{\ms}{\par\medskip}
\newcommand{\sk}{\par\smallskip}
\newcommand{\nin}{\par\noindent}
\newcommand{\q}{\quad}
\newcommand{\h}{\hbox}
\newcommand{\bl}{\bigl}
\newcommand{\br}{\bigr}
\newcommand{\msum}{\hbox{$\sum$}}
\newcommand{\motim}{\hbox{$\bigotimes$}}
\newcommand{\mopl}{\hbox{$\bigoplus$}}
\newcommand{\mcup}{\hbox{$\bigcup$}}
\newcommand{\ssc}{\,\raise.2ex\hbox{${\scriptstyle\circ}$}\,}
\newcommand{\ssb}{\raise.2ex\hbox{${\scriptscriptstyle\bullet}$}}
\newcommand{\eps}{\varepsilon}
\newcommand{\psim}{{}^{\mathfrak m}\psi}
\newcommand{\phim}{{}^{\mathfrak m\!}\varphi}
\newcommand{\rd}{\partial}
\newcommand{\C}{{\mathbb C}}
\newcommand{\N}{{\mathbb N}}
\newcommand{\Z}{{\mathbb Z}}
\newcommand{\R}{{\mathbb R}}
\newcommand{\Q}{{\mathbb Q}}
\newcommand{\DD}{{\mathbb D}}
\newcommand{\PP}{{\mathbb P}}
\newcommand{\A}{{\mathcal A}}
\newcommand{\Cc}{{\mathcal C}}
\newcommand{\D}{{\mathcal D}}
\newcommand{\I}{{\mathcal I}}
\newcommand{\Hc}{{\mathcal H}}

\newcommand{\M}{{\mathcal M}}
\newcommand{\Oc}{{\mathcal O}}
\newcommand{\Lo}{{}\,\overline{\!L}{}}
\newcommand{\Mo}{{}\,\overline{\!M}{}}
\newcommand{\Yo}{\overline{Y}}
\newcommand{\Sh}{\widehat{S}}
\newcommand{\Yh}{\widehat{Y}}
\newcommand{\Bt}{\widetilde{B}}
\newcommand{\Dt}{\widetilde{D}}
\newcommand{\Kt}{\widetilde{K}}
\newcommand{\Pt}{\widetilde{\PP}}
\newcommand{\St}{\widetilde{S}}
\newcommand{\Yt}{\widetilde{Y}}
\newcommand{\Wt}{\widetilde{W}}
\newcommand{\Spe}{{\mathcal S}pecan}
\newcommand{\Sp}{{\rm Sp}}
\newcommand{\Gr}{{\rm Gr}}
\newcommand{\can}{{\rm can}}
\newcommand{\Var}{{\rm Var}}
\newcommand{\Ker}{{\rm Ker}}
\newcommand{\Coker}{{\rm Coker}}
\newcommand{\Imm}{{\rm Im}}
\newcommand{\mon}{{\rm mon}}
\newcommand{\MHM}{{\rm MHM}}
\newcommand{\DR}{{\rm DR}}
\newcommand{\supp}{{\rm supp}}
\newcommand{\id}{{\it id}}
\newcommand{\pt}{{\it pt}}
\newcommand{\eq}{\,{=}\,}
\newcommand{\defs}{\,{:=}\,}
\newcommand{\simto}{\,\,\rlap{\hskip1.5mm\raise1.4mm\hbox{$\sim$}}\hbox{$\longrightarrow$}\,\,}
\newcommand{\simot}{\buildrel\sim\over\longleftarrow}
\newcommand{\into}{\hookrightarrow}
\newcommand{\ges}{\geqslant}
\newcommand{\les}{\leqslant}

\makeatletter
\renewcommand\section{\@startsection{section}{1}{0pt}{-3ex plus -1ex minus -.2ex}{2.3ex plus.2ex}{\centering\normalfont\bfseries}}
\makeatother
\theoremstyle{plain}
\newtheorem{thm}{Theorem}[section]
\newtheorem{cor}[thm]{Corollary}
\newtheorem{prop}[thm]{Proposition}
\newtheorem{lem}[thm]{Lemma}

\theoremstyle{definition}
\newtheorem{rem}[thm]{Remark}

\begin{document}
\title{Thom-Sebastiani Theorem for Hodge Modules}
\author{Morihiko Saito}
\address{RIMS Kyoto University, Kyoto 606-8502 Japan}
\begin{abstract}
We give a proof of the Thom-Sebastiani theorem for mixed Hodge modules using a compatibility with Verdier specialization.
\end{abstract}
\maketitle
\section*{Introduction}
Let $f_1$, $f_2$ be holomorphic functions on complex manifolds $X_1$, $X_2$ respectively.
Set $f=f_1+f_2$ on $X:=X_1{\times}X_2$.
The Thom-Sebastiani theorem describes the Milnor cohomology of $f$ in terms of those for $f_1$ and $f_2$.
This was proved by M.~Sebastiani and R.~Thom \cite{ST} in the isolated singularity case, by M.~Oka \cite{Ok} in the weighted homogeneous case, and by K.~Sakamoto \cite{Sak} in the general case.
\sk
This theorem has been generalized further by many people; the motivic theory of J.~Denef and F.~Loeser \cite{DL} for instance implies the local and global Thom-Sebastiani type theorems for the spectrum, that is, for the equivariant (limit) mixed Hodge numbers, see also \cite{Lo}, \cite{Ne}, \cite{NS}, \cite{SS}, \cite{Va} among others.
The theorem for bounded sheaf complexes with constructible cohomology sheaves has been shown in \cite{Mas}, \cite{Sch}, see also Theorem~\ref{T2.5} in this paper.
\sk
Forgetting the weight filtration, the theorem for the underlying filtered $\D$-modules of mixed Hodge nodules has been proved in \cite{MSS}.
We give a proof of the theorem with the weight filtration by employing a compatibility with Verdier specialization \cite{Ve}, which is based on discussions with P.~Deligne, see Theorems~\ref{T5.3} and also \ref{T2.5}.
There is however a certain twist between the two Thom-Sebastiani isomorphisms for the underlying $\Q$-complexes and $\D$-modules (see Section~\ref{S5}), and this makes the argument rather complicated.
\sk
The first version of this manuscript was written after the discussions with Deligne rather rapidly during the author's stay at the Institute for Advanced Study in 1988-89, and was typeset at RIMS in 1990 by using Plain Tex.
The first version also contained arguments about the nearby cycles, which were rather technical, and were omitted in this version to simplify the explanations.
We thank P.~Deligne for useful discussions, and the secretaries at RIMS for their excellent typing of the first version.
We also thank M.~Kontsevich and Y.~Soibelman for asking us about this long forgotten manuscript (see also \cite{KS}).

\tableofcontents
\numberwithin{equation}{section}

\section{Specialization of A-complexes} \label{S1}

In this paper analytic spaces are assumed to be separated and reduced.
We denote by $S$ the standard one-dimensional complex plane $\Spe\,\C[s]$, where standard means that $S$ has the natural coordinate $s$.
In Section~\ref{S1}, $A$ is a commutative noetherian ring.
\sk
Let $X$ be a closed subspace (not necessarily reduced) of a complex analytic space $Y$, and $p:D_XY\to S$ the deformation of $Y$ to the normal cone $C_XY$, i.e.
\begin{equation} \label{1.1}
\begin{aligned}
D_XY&:=\Spe_Y\bl(\mopl_{j\in\Z}\,\I^j_X\otimes s^{-j}\br),\\ C_XY&:=\Spe_X\bl(\mopl_{j\ges 0}\,\I^j_X/\I^{j+1}_X\br)=p^{-1}(0),
\end{aligned}
\end{equation}
where $\I_X$ is the ideal of $X,\I^j_X=\Oc_X$ for $j\les 0$, $p$ is the natural projection of $D_XY$ onto $S=\Spe\,\C[s]$, and $0$ is the origin of $S$.
(In general, $\Spe_Y\A$ for an $\Oc_Y$-algebra $\A$ of locally finite type (which is quasi-coherent over $\Oc_Y$) is defined by taking a locally defined surjection from a polynomial ring to $\A$ since the kernel is locally finitely generated.)
\sk
Let $q:D_XY\to Y$ be a natural morphism, and $q'$ its restriction to $p^{-1}(S^*)$, where $S^*=S\setminus\{0\}$.
Then we have a canonical isomorphism
\begin{equation*}
p{\times}q':p^{-1}(S^*)\buildrel\sim\over\longrightarrow S^*{\times}Y.
\end{equation*}
The Verdier specialization \cite{Ve} is defined by
\begin{equation*}
\Sp_XK=\psi_p q^*K\qquad{\rm for}\quad K\in D(A_Y),
\end{equation*}
see \cite{D3} for the definition of $\psi_p$.
\sk
Let $(Y{\times}S)'$ be the blow-up of $Y{\times}S$ along $X{\times}\{0\}$. Note that $D_XY$ is the affine open subset of $(Y{\times}S)'$ over $Y$ whose complement is the proper transform of $Y{\times}\{0\}$.
\ms
Luc Illusie informed us that the $\ell$-adic version of the next proposition was used in Gabber's proof of the local purity of the intersection complexes mentioned in \cite{BBD}.

\begin{prop} \label{P1.1}
Let $f_a:Y_a\to S\,\,(a=1,2)$ be a holomorphic function on a complex analytic space $Y_a$.
Put $X_a=f_a^{-1}(0)$.
Let $Y=Y_1{\times}_S Y_2$ with $\delta':Y\to Y_1{\times}Y_2$, $f:Y\to S$ canonical morphisms so that $X:=X_1{\times}X_2=f^{-1}(0)$.
Then, for $K_a\in D^b_c(A_{Y_a})\,(a=1,2)$, there is a canonical isomorphism in $D_c(A_X)$
\begin{equation} \label{1.2}
\psi_{f_1}K_1\boxtimes\psi_{f_2}K_2\simto\psi_f\delta'^* (K_1\boxtimes K_2).
\end{equation}
\end{prop}

\begin{proof}
Let $\rho:\St^*\to S^*$ be a universal covering.
Set
\begin{equation*}
\Yt=Y{\times}_S\St^*,\q\Yt_a=Y_a{\times}_S\St^*.
\end{equation*}
There is a commutative diagram
\begin{equation} \label{1.3}
\begin{gathered}
\xymatrix{\Yt \ar[r]^-{\tilde\delta'}\ar[d]_{\tilde j}& \Yt_1{\times}\Yt_2 \ar[d]^{\tilde j_1{\times}\tilde j_2}\\ 
Y \ar[r]^-{\delta'}\ar[d]_{f}& Y_1{\times}Y_2 \ar[d]^{f_1{\times}f_2}\\ 
S \ar[r]^-{\delta}& S{\times}S}
\end{gathered}
\end{equation}
where $\delta:S\to S{\times}S$ is the diagonal embedding, and $\tilde\delta'$ is a canonical morphism.
Put
\begin{equation*}
\Kt_a=\tilde j^*_aK_a,\q\Kt=\tilde j^*\delta'^*K\q\hbox{with}\q K=K_1\boxtimes K_2.
\end{equation*}
There are canonical morphisms
\begin{equation} \label{1.4}
\tilde j_{1*}\Kt_1\boxtimes\tilde j_{2*}\Kt_2\to (\tilde j_{1}{\times}\tilde j_{2})_*(\Kt_1\boxtimes\Kt_2)\to\delta'_*\tilde j_*\Kt,
\end{equation}
where the last morphism is induced by \eqref{1.3} and the restriction morphism $\id\to\tilde\delta'_*\tilde\delta'{}^*$.
The morphism \eqref{1.2} is obtained by applying $i^*_1{\times}i^*_2$ to \eqref{1.4}.
Here $i_a:X_a\into Y_a$ denotes a natural inclusion so that we have by definition
\begin{equation*}
\psi_{f_a}=i^*_a\tilde j_{a*}\tilde j^*_a,
\end{equation*}
(similarly for $\psi_f$), see \cite{D3}.
Then we get the isomorphism at each point of $X=X_1{\times}X_2$, by using the following lemma applied to $f_a$.
\end{proof}

\begin{lem} \label{L1.2}
Let $f:Y\to S$ be a holomorphic function on a complex analytic space $X$.
For $x\in X:=f^{-1}(0)$, let $B_\eps (x)$ be the $\eps$-ball with center $x$ which is defined by choosing a local embedding into a complex manifold with coordinates.
Then, for $K\in D^b_c(A_Y)$, there is $\eps_0>0$ together with $\eta_{\eps}>0$ for any $\eps\in(0,\eps_0)$ such that $\eta_{\eps}\ll\eps$ and the restriction of
\begin{equation*}
R^i(f|_{B_\eps (x)})_*(K|_{B_\eps (x)})\q\hbox{to}\q\,\Delta^*_{\eta_{\eps}}:=\bl\{s\in S\,\big|\, 0<|s|<\eta_{\eps}\br\}
\end{equation*}
is a local system whose stalks are independent of $\eps$ by the restriction morphisms, and are isomorphic to $\Hc^i(\psi_fK)_x$ {\rm (}this isomorphism is unique for $s\in\R_+\cap\Delta^*_{\eta_{\eps}}$ if a continuous lift of $\R_+\subset S^*$ to $\St^*$ is chosen{\rm )}.
\end{lem}

\begin{proof}
This is a direct consequence of the local existence of a Whitney stratification satisfying the Thom $A_f$-condition by Hironaka \cite{Hi} and the Thom isotopy theorem.
This finishes the proof of Lemma~\ref{L1.2} and Proposition~\ref{P1.1}.
\end{proof}

\begin{prop} \label{P1.3}
Let $Y_a$ be a complex analytic space, and $X_a$ its closed subspace for $a=1,2$.
Put $Y=Y_1{\times}Y_2$, $X=X_1{\times}X_2$, so that
\begin{equation*}
C_X Y=C_{X_1}Y_1{\times}C_{X_2}Y_2.
\end{equation*}
For $K_a\in D^b_c(A_{Y_a})\,(a=1,2)$, there is a canonical isomorphism in $D_c(A_{C_X Y})$
\begin{equation} \label{1.5}
\Sp_{X_1}K_1\boxtimes\Sp_{X_2}K_2\simto\Sp_X(K_1\boxtimes K_2).
\end{equation}
\end{prop}

\begin{proof}
With the notation of Section~\ref{S1}, we have
\begin{equation*}
D_X Y=D_{X_1}Y_1{\times}_S D_{X_2}Y_2,
\end{equation*}
i.e.
$p:D_XY\to S$ is the fiber product of $p_1:D_{X_1}Y_1\to S$ and $p_2:D_{X_2}Y_2\to S$.
We apply Proposition~\ref{P1.1} to $p_1$, $p_2$ and $q_1^*K_1$, $q_2^*K_2$, where $q_a:D_{X_a}Y_a\to Y_a$.
Then the assertion follows by using the commutative diagram
\begin{equation*}
\xymatrix{D_XY \ar[r]^-{\delta'}\ar[d]_{q}& D_{X_1}Y_1{\times}D_{X_2}Y_2 \ar[d]^{q_1{\times}q_2}\\
Y \ar@{=}[r] & Y_1{\times}Y_2}
\end{equation*}
This finishes the proof of Proposition~\ref{P1.3}.
\end{proof}

\subsection{Compatibility with a projection} \label{S1.1}

For $a=1,2$, let $\pi_a:Y_a\to X$ be a trivial vector bundle of rank $a$ on a complex analytic space $X$ with trivialization given by $\rho_a:Y_a\to\C^a$ so that we have the isomorphism
\begin{equation*}
\pi_a{\times}\rho_a:Y_a\simto X{\times}\C^a.
\end{equation*}
Let
\begin{equation*}
\pi:Y_2\to Y_1
\end{equation*}
be a morphism of vector bundles induced by a projection $\C^2\to\C$.
Using coordinates $s_1,s_2$ and $s_1$ of $\C^2$ and $\C$ respectively, this can be expressed as
\begin{equation*}
\pi:Y_2=X{\times}\C^2\ni(x,s_1,s_2)\mapsto(x,s_1)\in Y_1=X{\times}\C.
\end{equation*}
Here $\pi^*s_1$ is denoted by $s_1$. Take $K\in D^b_c(A_{Y_2})$, and set
\begin{equation} \label{1.6}
X_0:=X\cap\supp\,K\q\hbox{with}\q\supp\,K:=\mcup_i\,\supp\,\Hc^iK,
\end{equation}
where $X$ is identified with the zero sections of $Y_2$.
(Recall that the support of a sheaf is defined to be the complement of the largest open subset on which the sheaf vanishes, and hence the $\supp\,\Hc^iK$ and $\supp\,K$ are closed analytic subsets.)
\sk
Assume the following two conditions:
\begin{equation} \label{1.7}
\h{The induced morphism $\supp\,K\to X$ is injective.}
\end{equation}
\begin{equation} \label{1.8}
\begin{aligned}
&\h{For any $x\in X_0$ there is a fundamental neighborhood system $\{B_{x,j}\}_{j\in\N}$}\\ &\h{of $x$ in $X$ and $\gamma,\eta_j>0$ such that the restriction of the higher}\\&\h{direct image $R^k(\rho_2|_{B_{x,j}{\times}\C^2})_*(K|_{B_{x,j}{\times}\C^2})$ to the open subset}\\&\q\q\q\q\q\bl\{(s_1,s_2)\in\C^2\,\big|\,|s_2|<\eta_j,\, |s_1|<\gamma|s_2|\br\}\subset\C^2\\&\h{is a local system whose stalk is independent of $j$ by the restriction}\\&\h{morphisms, where $(s_1,s_2)$ are the coordinates of $\C^2$ as above.}
\end{aligned}
\end{equation}
Condition \eqref{1.7} is trivially satisfied if the support of $K$ is contained in a section of the projection $Y_2\to X$. Condition \eqref{1.8} is satisfied in the Thom-Sebastiani situation, where $B_{x,j}$ is the product of Milnor balls for $f_1,f_2$ in Lemma~\ref{L1.2} with $\gamma\eq1$. (Indeed, we have local systems on the complement of $\{t_1t_2\eq0\}$ locally, and $s_1\eq t_1{+}t_2$, $s_2\eq t_1{-}t_2$.) Set
\begin{equation} \label{1.9}
X_0^{(a)}:=X_0{\times}\C^a\q(a\eq1,2).
\end{equation}
This is identified with the restriction of $C_XY_a=X{\times}\C^a$ over $X_0\,(a=1,2)$.
We denote also by $\pi $ the induced morphisms $C_XY_2\to C_XY_1$ and $Y_2=X{\times}\C^2\to Y_1=X{\times}\C$.
These two morphisms are canonically identified to each other by using the vector space structures on $Y_1,Y_2$.

\begin{prop} \label{P1.4}
With the above assumptions, we have the canonical isomorphism
\begin{equation} \label{1.10}
(\Sp_X\pi_*K)|_{X_0^{(1)}}\to(\pi_*\Sp_XK)|_{X_0^{(1)}}.
\end{equation}
\end{prop}

\begin{proof}
Consider the morphism induced by $\pi{\times}\id$:
\begin{equation*}
p^{-1}_2(S^*)=Y_2{\times}S^*\to p^{-1}_1(S^*)=Y_1{\times}S^*.
\end{equation*}
This is naturally extended to $\pi:D_XY_2\to D_XY_1$ using the relative coordinate systems $(\tfrac{s_2}{s},\tfrac{s_1}{s},s)$ and $(\tfrac{s_1}{s},s)$ respectively of $D_XY_2$ and $D_XY_1$ over $X$.
Its base change by $\St^*\to S$ is denoted by $\tilde\pi:\Dt_X Y_2\to\Dt_XY_1$.
We have a commutative diagram
\begin{equation*}
\xymatrix{C_XY_2 \ar[r]^{i_2}\ar[d]^{\pi}& D_XY_2 \ar[d]^{\pi}&\Dt_XY_2 \ar[l]_{\tilde j_2}\ar[d]^{\tilde\pi}\\ 
C_XY_1 \ar[r]^{i_1}& D_XY_1 & \Dt_XY_1 \ar[l]_{\tilde j_1}}
\end{equation*}
Since $\tilde{j}_1^*\ssc\pi_*=\tilde{\pi}_*\ssc\tilde{j}_2^*$, the morphism \eqref{1.10} is then induced by the canonical morphism
\begin{equation*}
i^*_1\ssc\pi_*\to\pi_*\ssc i^*_2.
\end{equation*}
\sk
Let $(Y_a{\times}S)'$ be the blow-up of $Y_a{\times}S$ along $X{\times}\{0\}$ for $a=1,2$.
Let $(Y_2{\times}S)''$ be the blow-up of $(Y_2{\times}S)'$ along the proper transform of $\pi^{-1}(X{\times}\{0\})$, which is isomorphic to $X{\times}\C$.
The proper transform $E_1$ of the exceptional divisor of the first blow-up is isomorphic to $X{\times}\Pt^2$ where $\Pt^2$ is the blow-up of $\PP^2$ along the point corresponding to the intersection with the proper transform of $\pi^{-1}(X{\times}\{0\})$.
The exceptional divisor $E_2$ of the second blow-up is isomorphic to $X{\times}\C{\times}\PP^1$. The pictures of local affine coordinates of $(Y_2{\times}S)'$, $(Y_2{\times}S)''$ are as below.

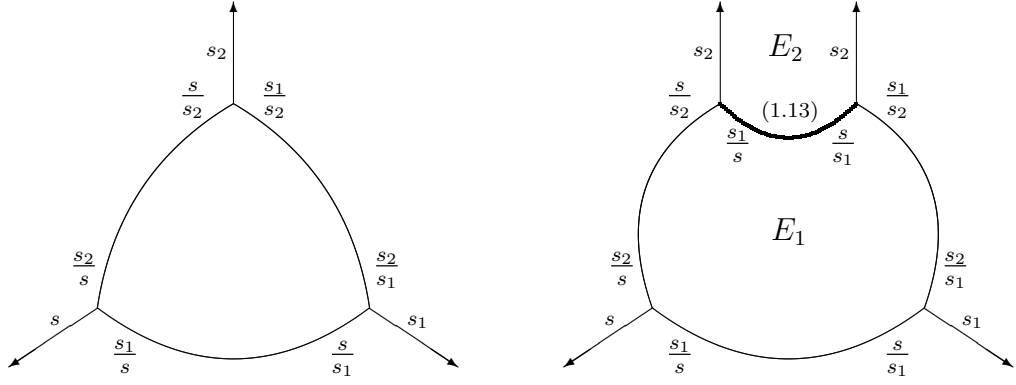
\hbox{$\setlength{\unitlength}{.9cm}
\begin{picture}(1,7)
\end{picture}$
$\setlength{\unitlength}{.9cm}
\begin{picture}(8,7)
\put(2,2){\vector(-3,-2){1.3}}
\put(6,2){\vector(3,-2){1.3}}
\put(4,5){\vector(0,1){1.5}}
\qbezier(2,2)(4,.5)(6,2)
\qbezier(6,2)(5.7,4)(4,5)
\qbezier(4,5)(2.3,4)(2,2)
\put(1.3,1.75){$\scriptstyle s$}
\put(1.6,2.5){$\frac{s_2}{s}$}
\put(2.2,1.2){$\frac{s_1}{s}$}
\put(6.55,1.75){$\scriptstyle s_1$}
\put(6.05,2.5){$\frac{s_2}{s_1}$}
\put(5.4,1.2){$\frac{s}{s_1}$}
\put(3.6,5.7){$\scriptstyle s_2$}
\put(4.4,5){$\frac{s_1}{s_2}$}
\put(3.2,5){$\frac{s}{s_2}$}
\end{picture}$
$\setlength{\unitlength}{.9cm}
\begin{picture}(8,7)
\put(2,2){\vector(-3,-2){1.3}}
\put(6,2){\vector(3,-2){1.3}}
\put(3,5){\vector(0,1){1.5}}
\put(5,5){\vector(0,1){1.5}}
\qbezier(2,2)(4,.5)(6,2)
\qbezier(6,2)(6.7,4)(5,5)
\qbezier(2,2)(1.3,4)(3,5)
\linethickness{.5mm}
\qbezier(3,5)(4,4)(5,5)
\put(1.3,1.75){$\scriptstyle s$}
\put(1.35,2.5){$\frac{s_2}{s}$}
\put(2.2,1.2){$\frac{s_1}{s}$}
\put(6.55,1.75){$\scriptstyle s_1$}
\put(6.25,2.5){$\frac{s_2}{s_1}$}
\put(5.4,1.2){$\frac{s}{s_1}$}
\put(2.6,5.7){$\scriptstyle s_2$}
\put(4.6,5.7){$\scriptstyle s_2$}
\put(5.4,5){$\frac{s_1}{s_2}$}
\put(2.2,5){$\frac{s}{s_2}$}
\put(3.05,4.3){$\frac{s_1}{s}$}
\put(4.6,4.3){$\frac{s}{s_1}$}
\put(3.75,3){$E_1$}
\put(3.7,5.7){$E_2$}
\put(3.6,4.8){$\scriptstyle\eqref{1.13}$}
\end{picture}$}

We see that the morphism $\pi$ naturally induces
\begin{equation*}
\pi':(Y_2{\times}S)''\to (Y_1{\times}S)'.
\end{equation*}
Its restrictions to $E_1=X{\times}\Pt^2$ and $E_2=X{\times}\C{\times}\PP^1$ are respectively induced by the projection $\Pt^2\to\PP^1$ associated to the blow-up of $\PP^2$ eliminating the point of indeterminacy of the rational map to $\PP^1$, and by the projection $\C\to pt$.
The closure of $\supp\,K{\times}S^*$ in $(Y_2{\times}S)''$ is proper over $(Y_1{\times}S)'$, since it is proper over $Y_1{\times}S$. Let
\begin{equation*}
p''_2:(Y_2{\times}S)''\to S,\q q''_2:(Y_2{\times}S)''\to Y_2,\q j:X_0^{(2)}\to p_2''{}^{-1}(0)
\end{equation*}
denote the canonical morphisms (see \eqref{1.9} for $X_0^{(2)}$).
Let $\overline{X}_0^{(2)}$ denote the closure of $X_0^{(2)}$ in $(Y_2{\times}S)''$. It is enough to show that we have the canonical isomorphism
\begin{equation} \label{1.11}
\psi_{p''_2}q''^*_2 K|_{\overline{X}_0^{(2)}\cap\pi'{}^{-1}(X_0^{(1)})}\simto j_*j^*\psi_{p''_2}q''^*_2 K|_{\overline{X}_0^{(2)}\cap\pi'{}^{-1}(X_0^{(1)})}.
\end{equation}
Indeed, the assertion follows from this by applying $\pi'_*$ to it and using the proper base change theorem.
Set
\begin{equation*}
\rd X_0^{(2)}:=\overline{X}_0^{(2)}\setminus X_0^{(2)}.
\end{equation*}
Then
\begin{equation*}
\overline{X}_0^{(2)}=X_0{\times}\Pt^2,\q\rd X_0^{(2)}=X_0{\times}(\Pt^2\setminus\C^2)\q\hbox{with}\q\Pt^2\setminus\C^2=\PP^1\cup\PP^1,
\end{equation*}
and the restriction of $\pi$ to $\overline{X}_0^{(2)}$ is induced by the projection $\Pt^2\to\PP^1$.
Note that the support of $\Sp_XK=\psi_{p''_2}q''^*_2 K|_{C_XY_2}$ is contained in $X_0^{(2)}=X_0{\times}\C^2$.
By the assumption \eqref{1.8} the restriction of $\Sp_XK$ to
\begin{equation*}
E_x:=\{x\}{\times}\{|s_1|<\gamma|s_2|\}\subset X{\times}\C^2=C_XY_2
\end{equation*}
is a local system.
\sk
For $z\in\rd X_0^{(2)}\cap\pi'{}^{-1}(X_0^{(1)})=X_0{\times}\C$, put $x=q''_2(z)$ and let $C_z$ be a circle in
\begin{equation*}
E_x\cap\pi'{}^{-1}\pi'(z)
\end{equation*}
around $z$.
Here note that $\pi'$ induces a bijection
\begin{equation*}
\rd X_0^{(2)}\cap\pi'{}^{-1}(X_0^{(1)})\simto X_0^{(1)}.
\end{equation*}
Then there is a canonical isomorphism
\begin{equation} \label{1.12}
(j_*j^*\psi_{p''_2}q''^*_2K)_z={\bf R}\Gamma (C_z,\Sp_XK).
\end{equation}
In particular the restriction of $j_*j^*\psi_{p''_2}q''^*_2K$ to
\begin{equation} \label{1.13}
\rd X_0^{(2)}\cap\pi'{}^{-1}(\{x\}{\times}\C)\simeq\C
\end{equation}
has constant cohomology sheaves.
\sk
On the other hand, for $z,x$ as above, set
\begin{equation*}
D^*_\eta=\{0<|s_2|<\eta,\,s_1=0\}\subset\C^2.
\end{equation*}
Then we have a canonical isomorphism for $j\gg 0$ and $0<\eta_j\ll 1$
\begin{equation} \label{1.14}
(\psi_{p''_2}q''^*_2K)_z={\bf R}\Gamma (B_{x,j}{\times}D^*_{\eta_j},K).
\end{equation}
Indeed, from relative affine coordinates $s_2,\,\tfrac{s_1}{s_2},\,\tfrac{s}{s_2}$ of $(Y_2{\times}S)'$, we get relative affine coordinates
\begin{equation*}
s_2,\,\,s''\,{:=}\,\tfrac{s}{s_2},\,\,s''_1\,{:=}\,\tfrac{s_1}{s}
\end{equation*}
of $(Y_2{\times}S)''$ around $\rd X_0^{(2)}\cap\pi'^{-1}(X_0^{(1)})$ over $X$.
So the left-hand side of \eqref{1.14} at $z$ is the limit of
\begin{equation*}
{\bf R}\Gamma\bl(B_{x,j}{\times}\bl\{(s_1,s_2)\,\big|\,\eta_j^{-1}c_j<|s_2|<\eta_j,\,|s_1{-}c''_1c_j|<\eta_jc_j\br\}, K\br)
\end{equation*}
for $j\to\infty$ with $0 < c_j\,(=s)<\eta_j\to 0$, where $z=x{\times}(0,0,c''_1)$ via the above relative affine coordinates $s_2,s'',s''_1$.
In particular, $\psi_{p''_2}q''^*_2K$ has also constant cohomology sheaves on \eqref{1.13}.
Thus the assertion is reduced to showing the isomorphism \eqref{1.11} at one point.
\sk
For $z'=x{\times}(0,c'',0)$ via the affine coordinates $s_2,s'',s''_1$ with $c''\ne0$ and for $0<c_j<\eta_j$ as above, we have the isomorphism
\begin{equation*}
(\Sp_XK)_{z'}={\bf R}\Gamma\bl(B_{x,j}{\times}\bl\{\bl(0,\tfrac{c_j}{c''}\br)\br\},K\br),
\end{equation*}
using the assumption \eqref{1.8}. So the assertion follows from \eqref{1.12} and \eqref{1.14}.
This finishes the proof of Proposition~\ref{P1.4}.
\end{proof}

\begin{thm} \label{T1.5}
Let $Y_a=X_a{\times}S$ be trivial line bundles on complex analytic spaces $X_a$ for $a\eq1,2$. Put $X\defs X_1{\times}X_2$, $\Yt\defs Y_1{\times}Y_2\eq X{\times}S{\times}S$, and $Y\defs X{\times}S$.
Let $\pi:\Yt\to Y$ be a morphism of vector bundles induced by $S{\times}S\ni (t_1,t_2)\mapsto t_1{+}t_2\in S$.
Let $K_a\in D^b_c(A_{Y_a})$. Put $Y'_a\defs\supp\,K_a$, $X'_a\defs X_a\cap Y'_a$, where $X_a$ is identified with the zero section of $Y_a$. Assume the composition $Y'_a\to Y_a\to X_a$ is injective for $a\eq1,2$.
Then we have a canonical isomorphism
\begin{equation} \label{1.15}
\Sp_X\pi_*(K_1\boxtimes K_2)|_C\simto\pi_*(\Sp_{X_1}K_1\boxtimes\Sp_{X_2}K_2)|_C\end{equation}
where $C=X'_1{\times}X'_2{\times}S$ is the restriction of $Y=C_XY$ over $X'_1{\times}X'_2$.
\end{thm}

\begin{proof}
This is a corollary of Lemma~\ref{L1.2}, Proposition~\ref{P1.3} and Proposition~\ref{P1.4}, where we apply Lemma~\ref{L1.2} to the projection $Y_a\to S$ for $a\eq1,2$.
\end{proof}

\section{Thom-Sebastiani theorem for A-complexes} \label{S2}
Let $\pi:Y\to X$ be a line bundle over a complex analytic space $X$, and $Y^*=Y\setminus X,\Yo=Y^*/\C^*$ with a projection $\bar\pi:Y^*\to\Yo $.
We say that $K\in D_c(A_Y)$ or $D_c(A_{Y^*})$ is {\it monodromical}, if the restrictions of $\Hc^iK$ to the fibers of $\bar\pi$ are locally constant.
Let $D_c(A_Y)_{\mon}$ and $D_c(A_{Y^*})_{\mon}$ be the full subcategory of $D_c(A_Y)$ and $D_c(A_{Y^*})$ respectively consisting of monodromical complexes.
\sk
Let $D_c(A_Y)_{\mon,!}$ (resp. $D_c(A_Y)_{\mon,*}$) be the full subcategory of $D_c(A_Y)_{\mon}$ defined by the condition : $i^*K=0$ (resp. $i^!K=0$), where $i:X\into Y$ is the zero section.
This condition is equivalent to that a natural morphism
\begin{equation*}
j_!j^*K\to K\q\h{(resp.}\,\,\,K\to j_*j^*K)
\end{equation*}
is an isomorphism, where $j:Y^*\to Y$.
So we get equivalences of categories
\begin{equation*}
D_c(A_Y)_{\mon,!}\simot D_c(A_{Y^*})_{\mon}\simto D_c(A_Y)_{\mon,*},
\end{equation*}
induced by $j_!$ and $j_*$.
The categories $D_c(A_Y)_{\mon}$, $D_c(A_{Y^*})_{\mon}$ are stable by the dual functor $\DD$, and the latter induces an equivalence of categories
\begin{equation*}
\DD:D_c(A_Y)_{\mon,*}\simto(D_c(A_Y)_{\mon,!})^{\rm op},
\end{equation*}
where $\Cc^{\rm op}$ denotes in general the opposite category of a category $\Cc$.
For $K\in D_c(A_Y)_{\mon}$ we have a natural isomorphism
\begin{equation*}
\pi_*K\simto i^*K.
\end{equation*}
We define the functors
\begin{equation*}
\mu:D_c(A_Y)_{\mon}\to D_c(A_Y)_{\mon,!},\q\nu:D_c(A_Y)_{\mon}\to D_c(A_Y)_{\mon,*},
\end{equation*}
by
\begin{equation} \label{2.1}
\mu(K)=C(\pi^*\pi_*K\to K),\q\nu(K)=C(K\to\pi^!\pi_!K)[-1],
\end{equation}
where $K$ is represented by an injective complex and the morphism in the mapping cone is induced by adjunction.
\sk
Let $M(A_Y;\mon)$ and $M(A_{Y^*};\mon)$ be the abelian category of $A_Y$-modules and $A_{Y^*}$-modules respectively whose restriction to the fibers of $\bar\pi:Y^*\to\Yo$ is locally constant.
Let $M(A_Y;\mon,!)$ be the full subcategory of $M(A_Y;\mon)$ defined by the condition $i^*K=0$ or equivalently $j_!j^*K=K$.
Let $D_c(A_{Y^*};\mon)$, $D_c(A_Y;\mon)$, $D_c(A_Y;\mon,!)$ be the full subcategory of their derived category consisting of the objects with constructible cohomology sheaves.
\sk
Let $M(A_X;T)$ be the abelian category of $A_X[T,T^{-1}]$-modules, and $D_c(A_X;T)$ be the full subcategory of the derived category of $M(A_X;T)$, whose objects have $A_X$-constructible cohomology sheaves.
If $\pi:Y\to X$ is a trivial line bundle with trivialization $p:Y\to S$, let $\rho:\Yt^*=Y^*{\times}_{S^*}\St^*\to Y^*$ and $\tilde\pi:\Yt^*\to X$ be natural morphisms, and $i_1:X\to Y$ a section of $\pi$ such that $\Imm(pi_1)=\{1\}\subset S$.
Then there are canonical isomorphisms
\begin{equation} \label{2.2}
\psi_pK\simot\tilde\pi_*\rho^*K\simto i_1^*K\q{\rm for}\q K\in D_c(A_{Y^*})_{\mon}.
\end{equation}
If $K\in M(A_{Y^*};\mon)$, $\tilde\pi_*\rho^*K$ and $i_1^*K$ have an action of the monodromy $T$, and \eqref{2.2} is compatible with the action of $T$.
So we get
\begin{equation*}
(\psi_pK,T)\simeq (i_1^*K,T)\q {\rm in}\q M(A_X;T)
\end{equation*}
for $K\in M(A_{Y^*};\mon)$.
They induce equivalences of categories
\begin{equation} \label{2.3}
M(A_{Y^*};\mon)\simto M(A_X;T),\q D_c(A_{Y^*};\mon)\simto D_c(A_X;T).
\end{equation}
Indeed, we get a functor $M(A_X;T)\to M(A_{Y^*};\mon)$ by taking the quotient of $\tilde\pi^*K'$ under the action of the monodromy for $K'\in M(A_X;T)$, and this gives a quasi-inverse of $\psi_p=i^*_1$.
Since $Y$ is a line bundle over $X$, we have a natural isomorphism $Y=C_XY$, and $Z:=D_XY$ is a trivial line bundle over $Y$ endowed with a projection $\pi':Z\to Y$ and a trivialization $p:Z\to S$.
The projection $\pi'$ is related with $q:Z\to Y$ in Section~\ref{S1} by
\begin{equation} \label{2.4}
q(z)=p(z)\cdot\pi'(z)\q\h{for}\,\,\,z\in Z,
\end{equation}
where the right-hand side is defined by using the natural $\C^*$-action on $Y$.
Applying \eqref{2.2} to $\pi'$ and $q^*K$, we get a canonical isomorphism
\begin{equation} \label{2.5}
\Sp_XK=K\q{\rm for}\q K\in D_c(A_Y)_{\mon}
\end{equation}
So we get an equivalence of categories
\begin{equation*}
D_c(A_Y;\mon)\simto D_c(A_Y)_{\mon}.
\end{equation*}
with quasi-inverse $\Sp$.
As a corollary, we get an equivalence of categories
\begin{equation} \label{2.6}
D_c(A_{Y^*};\mon)=D_c(A_Y;\mon,!)\simto D_c(A_Y)_{\mon,!}= D_c(A_{Y^*})_{\mon},
\end{equation}
since $D_c(A_Y;\mon,!)$ and $D_c(A_Y)_{\mon,!}$ are equivalent respectively to the full subcategories of $D_c(A_Y;\mon)$ and $D_c(A_Y)_{\mon}$ such that the restriction of their objects to $X$ is acyclic.
If $Y$ is a trivial line bundle with trivialization $p:Y\to X$, we have equivalences of categories
\begin{equation} \label{2.7}
D_c(A_{Y^*};\mon)\simto D_c(A_{Y^*})_{\mon}\simto D_c(A_X;T),
\end{equation}
by \eqref{2.3}, \eqref{2.6}, where the last functor is defined by $\psi_p$.

\begin{lem} \label{L2.1}
Let $Y$ be a complex analytic space with $f:Y\to S$ a holomorphic function.
Set $X=f^{-1}(0)$, $Y'=C_XY$. Let $f':Y'\to S$ the induced morphism by $f$.
Then we have $\Sp_XK\in D_c(A_{Y'})_{\mon}$ for $K\in D_c(A_Y)$ and there are natural isomorphisms in $D_c(A_X;T):$
\begin{equation} \label{2.8}
(\psi_fK,T)\simeq (\psi_{f'}\Sp_XK,T),\end{equation}
\begin{equation} \label{2.9}
(\varphi_fK,T)\simeq (\varphi_{f'}\Sp_XK,T)\simeq (\varphi_{f'}\mu(\Sp_XK),T)\buildrel{\sim}\over\longleftarrow (\psi_{f'}\mu(\Sp_X K),T),
\end{equation}
where the last isomorphism is induced by the morphism $\,{\rm can}$.
\end{lem}

\begin{proof}
It is enough to show \eqref{2.8} and the first isomorphism of \eqref{2.9}.
The function $f'$ is naturally extended to a function $g$ on $D_XY$ such that its restriction to $Y{\times}S^*$ is $s^{-1}f$, see \eqref{1.1}.
It is enough to show the commutativity of $\psi_{p}$ with $\psi_{g}$, $\varphi_{g}$ on $q^*K$, since $\psi_{g}q^*K$, $\varphi_{g}q^*K$ are monodromical with respect to the projection $X{\times}S\to X$, and their restrictions to $X{\times}\{s=1\}$, which are isomorphic to $\psi_fK$, $\varphi_fK$, are isomorphic to $\psi_{p}\psi_{g}q^*K$, $\psi_{p}\varphi_{g}q^*K$ by Section~\ref{S2}.
We consider the cartesian product of $\{0\}\to S\gets\St^*$ with itself, and take the base change by $g{\times}p$.
Then we get the canonical isomorphisms \eqref{2.8}--\eqref{2.9} by using Lemma~\ref{L1.2}.
The detail is left to the reader.
\end{proof}

\begin{prop} \label{P2.2}
Let $X_a,Y_a$ for $a=1,2$ and $\pi:\Yt\to Y$ be as in Theorem~{\rm \ref{T1.5}}.
Take $K_a\in D^b_c(A_{Y_a})_{\mon,!}$.
Then $\pi_*(K_1\boxtimes K_2)\in D_c(A_Y)_{\mon,!}$, and we have a canonical isomorphism in $D_c(A_X;T):$
\begin{equation} \label{2.10}
(\psim_p\pi_*(K_1\boxtimes K_2),T)\simeq (\psim_{p_1}K_1,T)\boxtimes(\psim_{p_2}K_2,T),
\end{equation}
where $p:Y\to S, p_a:Y_a\to S$ are natural projections, $\psim:=\psi[-1]$, and $T$ on the right-hand side is defined by $T\boxtimes T$.
\end{prop}

\begin{proof}
Let $\sigma:\Yt'\to\Yt$ be the blow-up along the zero section, and $i_E:E\eq X{\times}{\PP}^1\to\Yt'$ the exceptional set.
Put $K\defs K_1\boxtimes K_2$, $K'\defs\sigma^*K$, $\pi'\defs\pi\ssc\sigma$, $p'\defs p\ssc\pi'$.
Then
\begin{equation*}
(\psim_p\pi_*K,T)\simeq\pi'_*(\psim_{p'}K',T)\simeq\pi'_*i_{E*}i^*_E(\psim_{p'}K',T),
\end{equation*}
using the compactifications of the fibers of $\pi$ and the cohomologically local triviality of $K$ along the fibers of $(Y_1\setminus X_1){\times}(Y_2\setminus X_2)\to X$.
Set
\begin{equation*}
D'_0=(p\ssc\pi)^{-1}(0),\,\,\,D'_1=X_1{\times}Y_2,\,\,\, D'_2=Y_1{\times}X_2.
\end{equation*}
Let $D_i$ be the intersection of $E$ with the proper transform of $D'_i$.
Set
\begin{equation*}
E':=E\setminus D_0,\,\,\,E'':=E'\setminus (D_1\cup D_2),
\end{equation*}
with $j:E'\into E$, $j':E''\into E'$ the inclusion morphisms.
Then
\begin{equation*}
i^*_E(\psim_{p'}K',T)\simto j_*j^*i^*_E(\psim_{p'}K',T).
\end{equation*}
Indeed, the assertion is reduced to the case $X=pt$ by the monodromical condition, and follows from \cite[4.7]{sconj}. We have moreover
\begin{equation*}
i^*_{E'}(\psim_{p'}K',T)\buildrel\sim\over\leftarrow j'_!j'^*i^*_{E'}(\psim_{p'}K',T),
\end{equation*}
by the assumption $K_a\in D^b_c(A_{Y_a})_{\mon,!}$.
Let $P\in {\PP}^1$, and $H$ be the subbundle of $Y'$ corresponding to $X{\times}P\subset E$ (where $\{P\}$ is denoted by $P$ to simplify the notation).
Let
\begin{equation*}
i_P:X{\times}P\into E,\q i_H:H\into Y,\q i'_H:H\into Y'
\end{equation*}
be natural inclusions.
Then for $H\not=D'_i\ (i=0,1,2)$ we have
\begin{equation*}
i^*_Pi^*_E(\psim_{p'}K',T)\simeq(\psim_{p'}i'^*_HK',T)\simeq (\psim_{p\pi}i^*_HK,T),
\end{equation*}
by the transversality of $H$ with $p',K'$.
Since
\begin{equation*}
(\psi_{p\ssc\pi}i^*_H(K_1\boxtimes K_2),T)\simeq(\psi_{p_1}K_1,T)\boxtimes(\psi_{p_2}K_2,T),
\end{equation*}
the assertion is reduced to the following:
\end{proof}

\begin{lem} \label{L2.3}
Let $\pi:Y=X{\times}S\to X$ be a trivial line bundle.
Put
\begin{equation*}
L=\bl\{s\in S\,\big|\,0<{\rm Re}\,s <1, {\rm Im}\,s=0\br\},\q P_a=\{s=a\}\,(a=0,1),
\end{equation*}
and $S'=S\setminus(P_0\cup P_1),Y'=X{\times}S'$.
Let $(K,T)\in D_c(A_Y;T)$, and assume that the restriction of $\Hc^kK$ to the fibers of $Y'\to X$ is locally constant and that to $X{\times}P_a\,(a=0,1)$ is $0$.
Then for $P\in L$, we have a canonical isomorphism
\begin{equation*}
\pi_*(K,T)\simeq i^*_P(K,T)[-1]\q\h{in}\,\,\,D_c(A_X;T),
\end{equation*}
where $i_P:X{\times}P\into Y=X{\times}S$ is a natural inclusion.
\end{lem}

\begin{proof}
Let $\Lo$ be the closure of $L$ in $S$ (i.e.
$\Lo\simeq [0,1])$).
Consider the inclusion morphisms
\begin{equation*}
j:X{\times}L\into X{\times}\Lo,\q i_L:X{\times}L\into Y,\q i_{\Lo}:X{\times}\Lo\into Y.
\end{equation*}
Put
\begin{equation*}
\pi'=\pi\ssc i_{\Lo},\q\pi''=\pi\ssc i_L,\q K'=i^*_{\Lo}K,\q K''=i^*_LK.
\end{equation*}
Then we have natural isomorphisms
\begin{equation*}
\pi_*(K,T)\simto\pi'_*(K',T)\simeq\pi''_!(K'',T),
\end{equation*}
\begin{equation*}
(K'',T)\buildrel\sim\over\gets\pi''^*\pi''_*(K'',T)\simto\pi''^*i^*_P(K,T),
\end{equation*}
and the assertion follows from
\begin{equation*}
\pi''_!\pi''^*i^*_P(K,T)=\pi''_!A_{X{\times}L}\otimes_Ai^*_P(K,T),
\end{equation*}
since $\pi''_!A_{X{\times}L}\simeq A_X[-1]$, where the isomorphism is given by an orientation of $L=]0,1[$.
\end{proof}

\begin{prop} \label{P2.4}
Let $Y_a=X_a{\times}S$, $\Yt=Y_1{\times}Y_2$, $Y=X{\times}S$, $\pi:\Yt\to Y$, $p_a:Y_a\to S$, $p:Y\to S$ be as in Proposition~{\rm \ref{P2.2}}.
Then for $K_a\in D^b_c(A_{Y_a})_{\mon}$, $\pi_*(K_1\boxtimes K_2)\in D_c(A_Y)_{\mon}$ and we have canonical isomorphisms in $D_c(A_X;T):$
\begin{equation*}
(\phim_p\pi_*(K_1\boxtimes K_2),T)\simeq(\phim_{p_1}K_1,T)\boxtimes (\phim_{p_2}K_2,T),
\end{equation*}
where $\psim=\psi[-1],\phim=\varphi [-1]$.
\end{prop}

\begin{proof}
The assertion follows from Proposition~\ref{P2.2} if $K_a\in D^b_c(A_{Y_a})_{\mon,!}$, since the morphisms $\can$ are isomorphisms in this case.
In general we may assume
\begin{equation*}
K_a=[K'_a\buildrel u_a\over\longrightarrow K''_a],
\end{equation*}
such that $K'_a\in D^b_c(A_{Y_a})_{\mon,!}$ and $\pi^*_a\pi_{a*}K''_a=K''_a$, where $\pi_a:Y_a=X_a{\times}S\to X_a$, and $[A\to B]$ means the single complex with ${\rm deg}A=0$.
Indeed, it is enough to put
\begin{equation*}
K'_a=\mu(K_a)\,\,(\h{see \eqref{2.1}}),\q K''_a=\pi^*_a\pi_{a*}K_a[1].
\end{equation*}
Here we may assume that $K''_a$ is $A$-flat by replacing $\pi_{a*}K_a[1]$ with its $A$-flat resolution, and then $K'_a$ is $A$-flat and belongs to $D_c(A_{Y_a};\mon,!)$ by using an $A$-flat resolution of $\tilde\pi_{a*}\rho*_a(K'_a|_{Y_a^*})$ (see the proof of \eqref{2.7}).
We have
\begin{equation*}
\pi_*(K'_1\boxtimes K''_2)=\pi_*(K''_1\boxtimes K'_2)=0.
\end{equation*}
This implies
\begin{equation*}
\phim_p\pi_*(K_1\boxtimes K_2)\simto\phim_p\pi_*(K'_1\boxtimes K'_2).
\end{equation*}
So the assertion follows.
\end{proof}
\sk
As a corollary of Theorem~\ref{T1.5} and Proposition~\ref{P2.4} we get the following.

\begin{thm} \label{T2.5}
Let $X_a$ be a complex analytic space with $f_a$ a holomorphic function on $X_a$ and $K_a\in D^b_c(A_{X_a})$ for $a=1,2$.
Put $X=X_1{\times}X_2$, $f=f_1{+}f_2$ on $X$, $K=K_1\boxtimes K_2$, and $X_0=f^{-1}_1(0){\times}f^{-1}_2(0)$ with $i_0:X_0\into X$.
Then we have canonical isomorphisms in $D_c(A_{X_0};T):$
\begin{equation*}
(\phim_fK,T)\simeq (\phim_{f_1}K_1,T)\boxtimes (\phim_{f_2}K_2,T).
\end{equation*}
Here we assume that $\varphi_{f_a-c_a}K_a=0$ for any $c_a\in\C^*$ replacing $X_a$ with a sufficiently small open neighborhood of $f_a^{-1}(0)\subset X_a$ for $a\eq1$ or $2$.
\end{thm}

\section{Specialization of mixed Hodge modules} \label{S3}

We assume $A$ is a subfield of $\R$ from non on in this paper. We will denote by $\MHM(X,A)$ the abelian category of polarizable $A$-mixed Hodge modules on an complex analytic space $X$, and $D^b\MHM(X,A)$ its derived category of bounded complexes.
\sk
Let $X$ be a closed submanifold of a complex manifold $Y$.
We define as in Section~\ref{S1}
\begin{equation*}
p:D_XY\to S,\q q:D_XY\to Y\q\h{with}\q p^{-1}(0)=C_XY.
\end{equation*}
Let $M$ be a regular holonomic $\D_Y$-module.
Let $\theta$ be a locally defined vector field of the form
\begin{equation*}
\msum_{i=1}^l\,y_i\rd_{y_i}
\end{equation*}
for local coordinates $y_1,\cdots, y_m$ such that $X=\{y_1=\cdots=y_l=0\}$.
Let $G$ is a subset of $\C$ such that $0\in G$ and the composition
\begin{equation*}
G\to\C\to\C /\Z
\end{equation*}
is bijective.
We denote by $V$ Kashiwara's filtration on $M$ along $X$ \cite{Ka} (see also \cite{Mal} in the one-codimensional case).
Here $V$ is indexed decreasingly and the action of $\theta$ on ${\rm Gr}_V^k M$ has the minimal polynomial whose roots are contained in $G+k$.
By definition $\Gr_V M$ is a graded $\Gr_V\D_Y$-module, where the filtration $V$ on $\D_Y$ is defined by
\begin{equation*}
V^i\D_Y=\{P\in\D_Y:P{\I}_X^j\subset {\I}_X^{i+j}\,\,\,\hbox{for any $j$}\}.
\end{equation*}
Then
\begin{equation*}
[\Oc_{C_XY}]:=\mopl_{j\ges 0}\,{\I}_X^j/{\I}_X^{j+1}\subset\Gr_V\D_Y,\q\Gr_V\D_Y\subset\pi_*\D_{C_XY},\,\,\,\h{where}\,\,\,\pi:C_XY\to X.
\end{equation*}
By definition
\begin{equation} \label{3.1}
\Sp_X M=\D_{C_XY}\motim_{\pi^{-1}\Gr_V\D_Y}\pi^{-1}\Gr_VM=\Oc_{C_XY}\motim_{\pi^{-1}[\Oc_{C_XY}]}\pi^{-1}\Gr_VM.
\end{equation}
In the case ${\rm codim}_YX=1$ and $C_XY$ is trivial (e.g.
$X$ is globally a principal divisor), $\D_X$ is a subring of $\Gr_V^0\D_Y$ using the trivialization.
We define the $\D_X$-modules:
\begin{equation*}
\psi_fM=\Gr_V^0M,\qquad\varphi_fM=\Gr_V^{-1}M
\end{equation*}
if $X=f^{-1}(0)$ for a smooth morphism $f:Y\to S$.
We can verify that $\Sp_XM,\psi_fM,\varphi_fM$ are independent of the choice of $G$.
These definitions are compatible with the corresponding functors on $D^b(\C_Y)$ by the de~Rham functor (see \cite{Ka}, \cite{Mal}, \cite{mhp}, \cite{dual} for $\psi,\varphi$).
For $\Sp_X$ we have
\begin{equation} \label{3.2}
\Sp_XM=\psi_p\Hc^1q^*M,
\end{equation}
where $\psi,\varphi$ on regular holonomic $\D$-modules correspond to $\psim=\psi[-1],\phim=\varphi[-1]$.
Let $\pi':D_XY\to Y$ be the natural projection, and $j:p^{-1}(S^*)\into D_XY $ the natural inclusion.
Then
\begin{equation*}
j_*j^*{\Hc}^1q^*M=\Oc_{D_XY}\motim_{\pi'{}^{-1}[\Oc_{D_XY}}\pi'{}^{-1}M[s,s^{-1}]
\end{equation*}
where $[\Oc_{D_XY}]:=\mopl_{j\in\Z}\,({\I}_X^{-j}\motim s^j)$ (see \eqref{1.1}), and the filtration $V$ of $j_*j^*{\Hc}^1q^*M$ along $p^{-1}(0)$ is given by
\begin{equation} \label{3.3}
V^i(j_*j^*{\Hc}^1q^*M)=\Oc_{D_XY}\motim_{\pi'{}^{-1}[\Oc_{D_XY}]}\pi'{}^{-1}\bl(\mopl_{j\in\Z}V^{i-j}M\otimes s^j\br).
\end{equation}
This implies \eqref{3.2}, since $\psi_p{\Hc}^1q^*M=\psi_pj_*j^* {\Hc}^1q^*M$.
\sk
Now let $\M\in\MHM(Y,A)$ be a mixed $A$-Hodge module. We have
\begin{equation*}
\Sp_X\M=\psi_p{\Hc}^1q^*\M\in\MHM(C_XY,A),
\end{equation*}
see \cite[2.30]{mhm}.
The functor $\Sp_X$ is exact and induces
\begin{equation*}
\Sp_X:D^b\MHM(Y,A)\to D^b\MHM(C_XY,A),
\end{equation*}
since the forgetful functor $\MHM(X,A)\to D^b_c(A_X)^{[0]}$ is exact and faithful, and commutes with $\Sp_X$. (Here $D^b_c(A_X)^{[0]}$ denotes the abelian full subcategory of $D^b_c(A_X)$ which is stable by the dual functor $\DD$ and is constructed in \cite{BBD}.)
Let $E_k$ be the standard unipotent variation of $A$-mixed Hodge structures of rank $k{+}1$ on $S^*$ such that the monodromy has only one Jordan block and $\Gr_i^WE_k$ is a constant variation of Hodge structure of rank one if $i=0,2,\cdots, 2k$ and 0 otherwise.
Here standard means that $E_k$ satisfy the following conditions:
\begin{equation} \label{3.4}
\begin{aligned}
&\h{The mixed Hodge structure at $s=1$ naturally splits over $A$,}\\&\h{and its graded quotients $(\Gr_{2i}^WE_k)_{s=1}$ are given isomorphisms}\\&\h{with $A(-i)$ in a compatible way with the action of $N$,}
\end{aligned}
\end{equation}
\begin{equation} \label{3.5}
\begin{aligned}
&\h{The Hodge filtration $F$ on $E_k$ is generated over $\Oc_{S^*}$ over its}\\&\h{intersection with ${\rm Ker}(s\rd_s)^k\subset\Gamma(S^*,\Oc_{S^*}(E_k))$, where $\Oc_{S^*}(E_k)$}\\&\h{is the underlying $\Oc_{S^*}$-module of $E_k$,}
\end{aligned}
\end{equation}

Note that \eqref{3.5} is equivalent to that the Hodge filtration $F$ of $E_k$ is obtained like nilpotent orbit, and implies that the limit mixed Hodge structure is naturally isomorphic to the mixed Hodge structure at $s=1$.
Let ${\bar A}$ be the algebraic closure of $A$ in $\C$ with $G={\rm Gal}\,({\bar A}/A)$. Let $\Lambda$ be the set of closed points of ${\rm Spec}\,(A[x,x^{-1}])$, that is, the set of monic irreducible polynomials $P\in A[x]$ such that $P(0)\neq 0$.
Then $\Lambda\simeq {\bar A}^*/G$, where $\lambda=\{a_1,\cdots, a_d\}=Ga_i\subset {\bar A}^*$ corresponds to $P=\prod_i(x-a_i)$.
For $\lambda\in\Lambda$, we denote by $E^\lambda$ the standard semisimple variation of $A$-mixed Hodge structure of rank $d=|\lambda|$, with monodromy $\lambda$, and of type $(0,0)$.
Here standard means that
\begin{equation} \label{3.6}
\begin{aligned}
&\h{The restriction of $E^{\lambda}$ to $s\eq1$ is endowed with an isomorphism to $A[T]/(P(T))$}\\&\h{compatible with the action of the monodromy $T$, where $P$ corresponds to $\lambda$}\\&\h{as above and $i=\sqrt {-1}$ is chosen to define the monodromy $T$.}
\end{aligned}
\end{equation}

We define $E_k^\lambda=E^\lambda\otimes E_k$, and the corresponding mixed Hodge module on $S^*$ will be denoted by $E_k^\lambda[1]$ for the compatibility with the forgetful functor (i.e.
$E_k^\lambda\in D^b\MHM(S^*)$ and $H^iE_k^\lambda=0$ for $i\neq 1$).
Then for a morphism $f:Y^*\to S^*$ and for $\M\in\MHM(Y^*,A)$, $\M\otimes f^*E_k^\lambda$ is globally well-defined on $Y^*$ and belongs to $\MHM(Y^*,A)$, since $\otimes f^*E_k^{\lambda}$ is defined by using the pull-back to the ambient spaces of the local closed embeddings of $Y^*$.
Let $E_{k,A}^{\lambda}$ denote the underlying $A$-local system of $E_k^{\lambda}$, and $\rho:\St^*\to S^*$ a universal covering where we choose a base point $\tilde 1\in\St^*$ over $1:=\{s=1\}\in S^*$.
Then $E_{k,A}^{\lambda}$ is canonically identified with a local subsystem of $\rho_*A_{\St^*}$.
Indeed, the local subsystems of $\rho_*A_{\St^*}$ correspond bijectively to the finite dimensional vector subspaces of
\begin{equation} \label{3.7}
{\Gamma(\St^*,\rho^*\rho_*A_{\St^*}})=(\rho_*A_{\St^*})_1= (\rho^*\rho_*A_{\St^*})_{\tilde 1}=\prod_{s'\in\rho^{-1}(1)}A_{s'}
\end{equation}
invariant by the action of monodromy $T$, and $E_{k,A}^{\lambda}$ corresponds to the subspace
\begin{equation*}
{\Gamma(\St^*,\rho^*\rho_*A_{\St^*}})_k^\lambda:={\rm Ker}\, N^{k+1}\subset{\Gamma(\St^*,\rho^*\rho_*A_{\St^*}})^{\lambda},
\end{equation*}
where
\begin{equation*}
{\Gamma(\St^*,\rho^*\rho_*A_{\St^*}})^{\lambda}:=\mcup_j{\rm Ker}\,P(T)^j\subset\Gamma(\St^*,\rho^*\rho_* A_{\St^*}),
\end{equation*}
with $P(T)=\prod_{\alpha\in\lambda}(T-\alpha)$, and $N=\log T_u$ with $T=T_sT_u$ the Jordan decomposition of the monodromy $T$.
Here the Tate twist $(-1)=\Z(2\pi i)^{-1}\otimes_\Z$ in the usual definition of $N$ is trivialized by choosing $i=\sqrt {-1}$.
For the above identification we use the isomorphisms in \eqref{3.4}, \eqref{3.6} together with the multiplicative structure of $\rho_*A_{\St^*}$ and ${\Gamma(\St^*,\rho^*\rho_*A_{\St^*}})$, see \cite[2.3]{dual}.
Here we choose a coordinate $\tilde s$ of $\St^*$ such that $\rho^* s=\exp(2\pi i\tilde s)$ and $\tilde 1=\{\tilde s=0\}$.
Then $\rho^{-1}(1)=\Z$, and the basis of $(\Gr_{2i}^WE_{k,A}(i)) _{s=1}$ in \eqref{3.4} corresponds to $\tau^i/i!$ for $0\les i\les k$.
Here $\tau\in\prod_{j\in\Z}A(1)$ is defined by $-(2\pi i)\otimes j\in A(1)$ at $j\in\Z$.
For $E_A^{\lambda}$ we define the isomorphism
\begin{equation*}
(E_A^{\lambda})_{s=1}={\Gamma(\St^*,\rho^*\rho_*A_{\St^*}}) _0^{\lambda},
\end{equation*}
so that $1\in {\bar A}={\bar A}[T]/(T-\alpha)$ corresponds, after the scalar extension by $\motim_A{\bar A}$, to the element of $\prod_{j\in\Z}{\bar A}$ defined by $\alpha^{-j}\in{\bar A}$ at $j\in\Z$.
Here we use \eqref{3.6}--\eqref{3.7} and the canonical isomorphism
\begin{equation*}
{\bar A}[T]/(P(T))\simto\mopl_{\alpha\in\lambda}\,{\bar A}[T]/(T-\alpha)\simeq\mopl_{\alpha\in\lambda}\,{\bar A}.
\end{equation*}

\begin{lem} \label{L3.1}
Let $f:Y\to S$ and $X=f^{-1}(0)$ be as in Lemma~{\rm\ref{L1.2}}.
For $\M\in\MHM(Y,A)$, we have a canonical decomposition compatible with the action of $T_s:$
\begin{equation*}
\psi_f\M\simeq\mopl_{\lambda\in\Lambda}\,\psi_{f,\lambda}\M\q\hbox{in\,\,\,\,$\MHM(X,A)$,}
\end{equation*}
such that the eigenvalues of $T_s$ on $\psi_{f,\lambda}\M$ are $\lambda$.
\end{lem}

\begin{proof}
This follows from the definition of $\psi_f\M$, see \cite{mhp}, \cite{mhm}.
\end{proof}

\begin{lem} \label{L3.2}
With the notation of Section {\rm\ref{S3}}, let $i:X\into Y$, $j:Y^*=Y\setminus X\into Y$ be natural inclusions.
Then we have a canonical morphism
\begin{equation} \label{3.8}
\begin{aligned}\
&\q\,{\Hc}^{-1}i^*j_*(j^*\M\otimes f^*E_k^{\lambda})\\ &\simeq\Ker\bl(\psi_{f,1}(j^*\M\otimes f^*E_k^{\lambda})\buildrel{N}\over\to\psi_{f,1}(j^*\M\otimes f^*E_k^{\lambda}(-1))\br)\\ &\simeq\Ker\bl(\psi_f\M\boxtimes\psi_tE_k^{\lambda}\to\psi_f\M\boxtimes\psi_tE_k^{\lambda}\oplus\psi_f\M\boxtimes\psi_tE_k^{\lambda}(-1)\br)\\ &\to\psi_{f,\lambda^{-1}}\M,\\ 
\end{aligned}
\end{equation}
where the morphism in the third term is given by
\begin{equation*}
(T_s\boxtimes T_s -id)\oplus (N\boxtimes id + id\boxtimes N).
\end{equation*}
Moreover {\rm \eqref{3.8}} is an isomorphism for $k\gg 0$ such that $N^{k+1}=0$ on $\psi_{f,\lambda^{-1}}\M$ locally on $X$, the action of $T_s, N$ on $\psi_{f,\lambda^{-1}}\M$ correspond to $T_s^{-1}, -N$ on $f^*E_k^{\lambda}$, and {\rm \eqref{3.8}} is compatible with the natural isomorphism on the underlying complex of $A$-modules
\begin{equation*}
\psi_f K=i^*j_*\rho_*\rho^*j^*K\buildrel\sim\over\leftarrow i^*j_*(j^*K\otimes\rho_*A_{{\tilde X}^*})
\end{equation*}
by the inclusion $E_{k,A}^\lambda\into\rho_*A_{\tilde S^*}$ in Section~{\rm\ref{S3}}, where the base change of $\rho:\St^*\to S^*$ by $f$ is also denoted by $\rho:\Yt^*\to Y^*$.
\end{lem}

\begin{proof}
The first isomorphism of \eqref{3.8} follows from
\begin{equation*}
i_*i^*j_*=C(j_!\to j_*)=i_*C(N:\psi_{f,1}\to\psi_{f,1}(-1))
\end{equation*}
in \cite[2.24]{mhm}.
For the second it is enough to show the isomorphism
\begin{equation*}
\psi_f(j^*\M\otimes f^*E_k^\lambda)=\psi_f\M\boxtimes\psi_tE_k^{\lambda}
\end{equation*}
compatible with the corresponding isomorphism on the underlying $A$-complexes (in particular, $T_s,N$ corresponds to $T_s\boxtimes T_s,\, N\boxtimes id +id\boxtimes N)$.
By definition of $E_k^{\lambda}$ (see Section~\ref{S3}) we may replace $E_k^{\lambda}$ by $E^{\lambda}$ or $E_k$.
The case of $E^{\lambda}$ is trivial, and the case of $E_k$ is proved in \cite{ext}.
The last morphism of \eqref{3.8} is obtained by the product of the canonical projections
\begin{equation*}
\begin{aligned}\
&\psi_f\M\to\psi_{f,\lambda^{-1}}\M,\q\hbox{see (3.3.l)},\\ &\psi_tE^{\lambda}=A[T]/(P(T))\buildrel{\rm Tr}\over\longrightarrow A\,\,\hbox{(defined by $Q(T)\mapsto\sum_{\alpha\in\lambda}Q(\alpha))$},\\ &\psi_t E_k=\mopl_{0\les i\les k}\,A(-i)\to A.\\ 
\end{aligned}
\end{equation*}
Then for the proof of the isomorphism, it is enough to show it for the underlying $\C$-complex, and we may forget $E^{\lambda}$ essentially, i.e.
replace $E_k^{\lambda}$ by $E_k$.
Then the assertion is proved in \cite{ext}.
\end{proof}

\begin{prop} \label{P3.3}
In the notation of Proposition~{\rm \ref{P1.3}}, we have a canonical isomorphism for $\M_a\in D^b\MHM(Y_a,A):$
\begin{equation} \label{3.9}
\Sp_{X_1}\M_1\boxtimes\Sp_{X_2}\M_2\simto\Sp_X(\M_1\boxtimes\M_2)
\end{equation}
compatible with {\rm \eqref{1.5}}.
\end{prop}

\begin{proof}
We may assume $\M_a\in\MHM(Y_a,A)$, since $\Sp$ and $\boxtimes$ are exact functors.
By the exactness and faithfulness of the forgetful functor, it is enough to construct a morphism \eqref{3.9} compatible with \eqref{1.5}.
We have the multiplications
\begin{equation*}
\delta^*(E_k^{\lambda}\boxtimes E_{k'}^{\lambda'}) =E_k^{\lambda}\otimes E_{k'}^{\lambda'}\to\mopl_{\lambda''\subset\lambda\lambda'}\, E_{k+k'}^{\lambda''}
\end{equation*}
compatible with $\rho_*A_{\St}\otimes\rho_*A_{\St^*}\to\rho_*A_{\St^*}$ by the inclusion $E_{k,A}^{\lambda}\into\rho_*A_{\St^*}$.
This induces the desired morphism by the argument in the proof of Proposition~\ref{P1.3}.
\end{proof}

\begin{prop} \label{P3.4}
With the notation of Subsection~{\rm\ref{S1.1}}, let $\M^{\ssb}\in D^b\MHM(Y_2,A)$ such that conditions {\rm \eqref{1.7}} and {\rm \eqref{1.8}} with $K$ replaced by the direct sum of the underlying $A$-complex of $\M^j$ for $j$. Assume there is a principal divisor $X_0'\subset X$ such that
\begin{equation*}
\mcup_j\pi({\rm supp}\,\M^j)\cap X'_0=X_0\,(:=\mcup_j{\rm supp}\,\M^j\cap X),
\end{equation*}
where $X$ is identified with the zero sections of $Y_1,Y_2$. Then we have a canonical isomorphism
\begin{equation} \label{3.10}
i_*i^*\Sp_X\pi_*\M^{\ssb}\to\pi_*\Sp_X\M^{\ssb}\q\h{in}\,\,\,D^b\MHM(C_XY_1,A)
\end{equation}
compatible with \eqref{1.10} in Proposition~{\rm\ref{P1.4}}, where $i:C_XY_1{\times}_XX'_0\into C_XY_1$.
\end{prop}

\begin{proof}
The proof of Proposition~\ref{P1.4} holds for the underlying $A$-complex of $\M^j$ with $X_0$ replaced by $X_0'$.
Then it is enough to construct a morphism \eqref{3.10} compatible with \eqref{1.10}.
The restriction $i^*$ is defined by $C(\psi_{g,1}\to\varphi_{g,1})$ for a defining function $g$ of $X_0'$. (Note that $g$ is given by $f_1{-}f_2$ in the application, since $s_2=t_1{-}t_2$.)
With the notation of the proof of Proposition~\ref{P1.4}, it is enough to define the direct image of $({\Hc}^1 q''^*_2)\M^j|_{\pi^{'-1}(D_{X_1}Y_1)}$ in $C^b\MHM(D_{X_1}Y_1,A)$ by the morphism $\pi':\pi'^{-1}(D_{X_1}Y_1)\to D_{X_1}Y_1$.
By definition $\pi'^{-1}(D_{X_1}Y_1)$ is isomorphic to the blow up of $D_{X_1}Y_1{\times}\C$ along $p_1^{-1}(0){\times}\{0\}$.
Indeed, using the coordinates $(s_1,s_2)$ of $\C^2$ as in Subsection~\ref{S1.1} and $s$ of $S,\>\pi'^{-1}(D_{X_1}Y_1)$ is covered by two affine opens over $X$ with relative coordinate systems $(s,\tfrac{s_1}{s},\tfrac{s_2}{s})$ and $(s_2,\tfrac{s_1}{s},\tfrac{s}{s_2})$, and $D_{X_1}Y_1$ is affine over $X$ with relative coordinate system $(s,\tfrac{s_1}{s})$.
Let $H_k$ be the divisor of $\pi'^{-1}(D_{X_1}Y_1)$ defined by $\tfrac{s_2}{s}=c_k$ for $c_k\in\C^*\> (k=1,2)$, and $U_k=\pi'^{-1}(D_{X_1}Y_1)\setminus H_k$.
Then $U_k$ are affine over $X$ and $D_{X_1}Y_1$, and the $H_k$ are transversal to $({\Hc}^1\tilde q''^*_2)\M^j$.
Indeed, the underlying $A$-complex of $({\Hc}^1\tilde q''^*_2)\M^j$ is locally constant along each orbit of the $\C^*$-action on $D_{X_2}Y_2$, which is defined by
\begin{equation*}
(s,\tfrac{s_1}{s},\tfrac{s_2}{s})\mapsto(\alpha s,\tfrac{1}{\alpha}\tfrac{s_1}{s},\tfrac{1}{\alpha}\tfrac{s_2}{s})\,\,\,\,\hbox{for}\,\,\,\alpha\in\C^*,
\end{equation*}
and is transversal to $H_k$.
Let $i_k:H_k\into\pi'^{-1}(D_{X_1}Y_1),\,j_k:U_k\into\pi'^{-1}(D_{X_1}Y_1)$ denote the inclusion morphisms.
Then the direct image of $({\Hc}^1q''^*)\M^j$ by $\pi''=\pi'|_{\pi'^{-1}(Z_1)}$ is defined by the functorial $\pi_*''$-acyclic resolution
\begin{equation} \label{3.11}
(i_1)_*(\Hc^{-1}i_1^*)({\Hc}^1q''^*)\M^j\to (j_1)_!j_1^*j_{2*}j_2^*({\Hc}^1q''^*)\M^j\to (i_2)_*({\Hc}^1i_2^!)({\Hc}^1q''^*)\M^j
\end{equation}
of $({\Hc}^1q''^*)\M^j$ as in \cite{B}.
By the commutativity of $\psi_{g,1},\varphi_{g,1},\psi_{p_2''}$ with ${\Hc}^{\ssb}\pi''_*$, we see that $\psi_{g,1}\psi_{p_2''}$ or $\varphi_{g,1}\psi_{p_2''}$ of \eqref{3.11} gives a $\pi''_*$-acyclic resolution of $\psi_{g,1}$ or $\varphi_{g,1}$ of $\psi_{p_2''}({\Hc}^1q''^*)\M^j$.
So we get the assertion.
\end{proof}

\begin{rem} \label{R3.5}
We can replace $j_*j^*\psi_{p''_2}$ of \eqref{3.11} for $({\Hc}^1q''^*)\M^{\ssb}$ by the $\pi''_*$-acyclic resolution
\begin{equation*}
(i_1)_*({\Hc}^{-1}i_1^*)\to (j_1)_!j_1^*
\end{equation*}
of $j_*j^*\psi_{p''_2}({\Hc}^1q''_2*)\M^{\ssb}=j_*\Sp_X\M^{\ssb}$, and the right-hand side of \eqref{3.10} can be defined in this way.
\end{rem}

\begin{thm} \label{T3.6}
With the notation of Theorem~{\rm \ref{T1.5}}, let $\M_a\in D^b\MHM(Y_a,A)$, and $Y'_a\subset Y_a$ the graph of a holomorphic function $f_a$ on $X_a$, i.e.
$Y'_a=\{f_a=t_a\}$.
If $\supp\,\M_a\subset Y'_a$, we have a canonical isomorphism in $D^b\MHM(Y,A):$
\begin{equation*}
i_*i^*\Sp_X\pi_*(\M_1\boxtimes\M_2)\simto\pi_*(\Sp_{X_1}\M_1\boxtimes\Sp_{X_2}\M_2)
\end{equation*}
compatible with {\rm \eqref{1.15}} in Theorem~{\rm\ref{T1.5}}, where $i$ denotes the inclusion $\{f_1\eq f_2\}{\times}S\into Y$.
\end{thm}

\begin{proof}
This is a corollary of Proposition~\ref{P3.3} and Proposition~\ref{P3.4}. Note that $i_*i^*$ in the left-hand side of the isomorphism can be defined by $C(\psi_{g,1}\to\varphi_{g,1})$ with $g=f_1{-}f_2$, since $s_2=t_1{-}t_2$.
\end{proof}

\section{Monodromical mixed Hodge modules} \label{S4}
Let $\pi:Y\to X$ be a line bundle on a complex manifold, and $\theta$ the Euler vector field on $Y$ corresponding to the natural $\C^*$-action.
Let $M$ be a regular holonomic $\D_Y$-module, and $V$ Kashiwara's filtration along $X$ the zero section of $Y$ as in Section~\ref{S3}.
We define $[M]=\mopl_{\alpha\in\C}\,(\pi_*M)^\alpha$ with
\begin{equation*}
(\pi_*M)^\alpha=\mcup_k\,{\rm Ker}((\theta-\alpha)^k:\pi_*M\to\pi_*M).
\end{equation*}
We define $[\D_Y]=\mopl_{i\in\Z}\,(\pi_*\D_Y)^i$ and $[\Oc_Y]=\mopl_{i\in\N}\,(\pi_*\Oc_Y)^i$ in the same way, where the action of $\theta$ is given by $\theta(P)=[\theta, P]$ for $P\in\D_Y$.
Then we have a natural isomorphism as graded rings:$[\D_Y]\simto\Gr_V\D_Y$, and $[M]$ is a graded $[\D_Y]$-module.
Put $Y^*=Y\setminus X$.
We say that $M$ is {\it monodromical}, if the following equivalent conditions are satisfied:
\begin{equation} \label{4.1}
\h{$\DR(M)$ is monodromical, see Section~\ref{S2},}
\end{equation}
\begin{equation} \label{4.2}
\h{$M|_{Y^*}$ is generated over $\Oc_{Y^*}$ by ${\rm Ker}(\theta:M|_{Y^*}\to M|_{Y^*})$,}
\end{equation}
\begin{equation} \label{4.3}
\h{$M$ is generated over $\Oc_Y$ by $\pi^{-1}[M]$,}
\end{equation}
\begin{equation} \label{4.4}
\begin{aligned}
\h{$M\buildrel\sim\over\leftarrow \D_Y\motim_{\pi^{-1}[\D_Y]}\pi^{-1}[M] \buildrel\sim\over\leftarrow \Oc_Y\motim_{\pi^{-1}[\Oc_Y]}\pi^{-1}[M]$,}
\end{aligned}
\end{equation}
\begin{equation} \label{4.5}
\begin{aligned}
&\h{$M\simeq \D_Y\motim_{\pi^{-1}[\D_Y]}\pi^{-1}\Mo\simeq \Oc_Y\motim_{\pi^{-1}[\Oc_Y]}\pi^{-1}\Mo$ for a coherent graded $[\D_Y]$-module}\\&\h{$\Mo=\mopl_{\alpha\in\C}\,\Mo^\alpha$ such that $\theta-\alpha$ is nilpotent on $\Mo^\alpha$ locally on $X$,}
\end{aligned}
\end{equation}
\begin{equation} \label{4.6}
\h{$M\simeq\Sp_XM$.}
\end{equation}
\ms
Here \eqref{4.4} $\Rightarrow$ \eqref{4.3} $\Rightarrow$ \eqref{4.2} $\Rightarrow$ \eqref{4.1} is easy using the projection $Y\setminus X\to X$ locally on $Y\setminus X$ and also the Riemann-Hilbert correspondence. We employ the latter for \eqref{4.1} $\Leftrightarrow$ \eqref{4.6}, and \eqref{3.1} for \eqref{4.6} $\Rightarrow$ \eqref{4.5}.
The remaining \eqref{4.5} $\Rightarrow$ \eqref{4.4} is verified by showing
\begin{equation} \label{4.7}
\Mo\simto[M]\simto\Gr_VM
\end{equation}
\nin
under the assumption \eqref{4.5}, where $\Gr_VM$ is naturally $\C$-graded by the action of $\theta$.
\sk
Let $\hat\pi:\Yh\to X$ be the natural projective compactification of $\pi$.
We say that a mixed Hodge module $\M\in\MHM(Y,A)$ is {\it monodromical}, if we have an isomorphism
\begin{equation} \label{4.8}
\M\simeq\Sp_X\M\q {\rm in}\q\MHM(Y,A)
\end{equation}
where $X$ (resp. $Y$) is naturally identified with the zero section of $\pi$ (resp. $C_XY$).
This condition implies
\begin{equation} \label{4.9}
\h{$\M$ is extendable over $\Yh$,}
\end{equation}
\ms\nin
by the blow-up construction, see Section~\ref{S1}.
\sk
Let $\MHM(Y,A)_{\rm mon}$ be the full subcategory of monodromical mixed Hodge modules.
Let $\MHM(Y,A)_{{\rm mon},*}$ (resp. $\MHM(Y,A)_{{\rm mon},!}$) be its full subcategory defined by the condition
\begin{equation*}
\M\simto j_*j^*\M\q\hbox{(resp. $j_!j^*\M\simto\M$)}
\end{equation*}
where $j:Y^*=Y\setminus X\to Y$.
Let $\MHM(X,A;T_s,N)$ be the category of Mixed Hodge modules on $X$ with commuting actions of $T_s,N$ such that $T_s:\M\to\M$ has finite order and $N:\M\to\M(-1)$ is nilpotent (locally on $X$).
If $Y$ is a trivial line bundle $X{\times}S$ with projection $p:Y\to S$, we have the vanishing cycle functors
\begin{equation} \label{4.10}
\psi_p,\varphi_p:\MHM(Y,A)\to\MHM(X,A;T_s,N)
\end{equation}
where $\varphi_p=\varphi_{p,1}\oplus\psi_{p,\neq 1}$ i.e.
$\psi_{p,\neq 1}=\varphi_{p,\neq 1}$, see \cite{mhm}.

\begin{lem} \label{L4.1}
Let $\pi:Y\to X$ be as in Section~{\rm\ref{S4}}.
For a filtered $\D$-module $(M,F)$, set
\begin{equation*}
\begin{aligned}
F_p[M]&:=\mopl_{\alpha\in\Q}\,F_p(\pi_*M)^\alpha\q\h{with}\q F_p(\pi_*M)^\alpha:=\pi_*F^pM\cap (\pi_*M)^\alpha,\\ F_p\Gr_VM&:=\mopl_{\alpha\in\Q}\,F_p\Gr_V^\alpha M.
\end{aligned}
\end{equation*}
Then, for $\M=((M,F),K;W)\in\MHM(Y,A)$, $\M$ is monodromical, if and only if the following equivalent conditions are satisfied
\begin{equation} \label{4.11}
\h{\it $F_pM$ are generated over $\Oc_Y$ by $\pi^{-1}F_p[M]$,}
\end{equation}
\begin{equation} \label{4.12}
\h{\it $F_pM\buildrel\sim\over\leftarrow \Oc_Y\motim_{\pi^{-1}[\Oc_Y]}\pi^{-1}F_p[M]$,}
\end{equation}
\begin{equation} \label{4.13}
\h{\it $F_p[M]\simto F_p\Gr_VM$ and $K$ is monodromical, see Section~{\rm\ref{S2}}.}
\end{equation}
\end{lem}

\begin{proof}
By \eqref{3.3} the filtration $V$ of $M$ is indexed by $\Q$ and \eqref{3.3} holds with $V^i,V^{i-j}$ replaced by $V^\alpha,V^{\alpha -j}$ for $\alpha\in\Q$.
Then the underlying filtered $\D_{C_XY}$-module $(M',F)$ is given by
\begin{equation*}
F_pM'=\Oc_{C_XY}\motim_{\pi'{}^{-1}[\Oc_{C_XY}]}\pi'{}^{-1}(F_p\Gr_VM)
\end{equation*}
and \eqref{4.11} follows from the monodromical condition.
We verify the equivalence \eqref{4.11} $\Leftrightarrow$ \eqref{4.12} using the flatness of $\Oc_Y$ over $\pi^{-1}[\Oc_Y]$ and \eqref{4.3} $\Leftrightarrow$ \eqref{4.4}.
Then \eqref{4.12} $\Rightarrow$ \eqref{4.13} is easy, where the second condition of \eqref{4.13} follows from \eqref{4.11} and \eqref{4.1} $\Leftrightarrow$ \eqref{4.3}.
Assume \eqref{4.13}.
Then we have the isomorphism \eqref{2.5}.
It is compatible with the weight filtration, since the monodromy of $\psi_pq^*\Gr_i^WK$ is semisimple.
Indeed, we have
\begin{equation} \label{4.14}
\begin{aligned}
&\h{A polarizable variation of Hodge structure on $S^*$ has a semisimple monodromy,}\\ &\h{and its Hodge filtration is given by local systems,}
\end{aligned}
\end{equation}
\ms\nin
since its pull back to a universal covering $\St^*\,\, (\simeq\C)$ is a constant variation.
We verify that \eqref{2.5} is compatible with the isomorphism
\begin{equation*}
M\buildrel\sim\over\leftarrow\Oc_Y\motim_{\pi^{-1}[\Oc_Y]}\pi^{-1}[M]\simto\Oc_Y\motim_{\pi^{-1}[\Oc_Y]}\Gr_VM\simeq\Sp_XM
\end{equation*}
using the splitting of \eqref{3.3} which gives an isomorphism $\psi_p{\Hc}^1q^*M\cong{\Hc}^1i_1^!{\Hc}^1q^*M$ where $i_1:Y{\times}\{1\}\to D_XY$.
Then by the first condition of \eqref{4.13} we get the morphism of mixed Hodge modules $\Sp_X((M,F),K;W)\to ((M,F),K;W)$ which must be an isomorphism by the strictness.
\end{proof}

\begin{rem} \label{R4.2}
As a corollary of the above proof, we get a canonical isomorphism
\begin{equation} \label{4.15}
\M\simeq\Sp_X\M\q{\rm for}\q\M\in\MHM(Y,A)_{\rm mon}
\end{equation}
where $\pi:Y\to X$ may be as in Section~\ref{S4}, since the assertion is local and we can replace $Y,X$ by smooth ones using local embeddings.
By a similar argument we get the canonical isomorphism compatible with \eqref{2.2} :
\begin{equation} \label{4.16}
\psi_p\M\simto{\Hc}^{-1}i_1^*\M\q {\rm for}\q\M\in\MHM(Y,A)_{\rm mon}
\end{equation}
where $Y$ is a line bundle $X{\times}S$ with projection $p:Y\to S$.
Indeed, we get the morphism \eqref{4.16} by \eqref{4.11}, \eqref{4.13}, and it must be an isomorphism by strictness.
Here note that with the notation of Section~\ref{S4}, ${\rm Ker}(\theta|_{Y^*})$ and $\pi^{-1}[M]|_{Y^*}$ are related by the Nilson class functions on $S^*$, and this relation is same as in the isomorphism between $\psim_p\DR$ and $\DR\psi_p$.
By construction the isomorphism \eqref{4.15} coincides with the one obtained by \eqref{4.16} and the lemma below.
We get also
\begin{equation*}
\Sp_X\M\in\MHM(C_XY,A)_{\rm mon}\q{\rm for}\q\M\in\MHM(Y,A),
\end{equation*}
where $X,Y$ are as in Section~\ref{S3}.
\end{rem}

\begin{prop} \label{P4.3}
With the notation of Section~{\rm\ref{S4}}, $\M\in\MHM(Y,A)$ is monodromical, if and only if its underlying $A$-complex is monodromical and {\rm \eqref{4.9}} holds.
\end{prop}

\begin{proof}
Assume \eqref{4.9} and $K\simeq\Sp_XK$, where $K$ is the underlying $A$-complex.
We first reduce to the case of trivial line bundle $Y=X{\times}S $.
Consider ${\Hc}^1q^*\M$ on $D_XY=Y{\times}S$, see \eqref{2.4}.
They satisfy the assumptions.
Indeed, $q{\times}p:D_XY\to Y{\times}S$ induces a bimeromorphic map between $\Yh{\times}\Sh$ and $\Yh{\times}\Sh$ by \eqref{2.4} (where $\Sh={\PP}^1$) and $ {\Hc}^1q^*\M$ is extendable to $\Yh{\times}\Sh$ by \cite[2.18]{mhm}, since $\M\boxtimes A_{h,S}[1]$ is extendable to $\Yh{\times}\Sh$.
(In this paper, the constant mixed Hodge module on a smooth variety $Z$ is denoted by $A_{h,Z}[\dim Z]$).
Then we can use \eqref{4.16}.
Secondly we reduce to the case $\M\in\MHM(Y,A)_{{\rm mon},*}$.
Let $j:Y^*\into Y$ and $i:X\into Y$ be natural inclusions.
Consider an exact sequence
\begin{equation*}
0\to i_*({\Hc}^0i^!)\M\to\M\to j_*j^*\M\to i_*({\Hc}^1i^!)\M
\end{equation*}
and put $\M'={\rm Im}\,(\M\to j_*j^*\M)$.
If $j_*j^*\M$ is monodromical, so is $\M'$, since $\Sp_X$ is exact and induces the identity on $i_*\MHM(X,A)$.
Moreover the morphism
\begin{equation*}
{\rm Ext}\/^1(\M',i_*({\Hc}^0i^!)\M)\to {\rm Ext}\/^1(\Sp_X\M',\Sp_Xi_*({\Hc}^0i^!)\M)
\end{equation*}
induced by $\Sp_X$ is an isomorphism.
Indeed, the left-hand side is naturally isomorphic to ${\rm Ext}\/^1(i_*i^*\M',i_*({\Hc}^0i^!)\M)$ and $\Sp_X$ commutes with $i_*i^*, j_!j^*$, see \cite[2.30]{mhm}.
Therefore \eqref{4.8} for $\M$ is reduced to that for $j_*j^*\M$.
Let $\check E_k^{\lambda}$ be the dual of $E_k^{\lambda}$.
Their pull-backs to $Y$ will be also denoted by $\check E_k^{\lambda},E_k^{\lambda}$.
Put $\pi'=\pi|_{Y^*}$.
Then for $\M'\in\MHM(Y^*,A)$ extendable to $\Yh$, we have a natural morphism of mixed Hodge modules
\begin{equation} \label{4.17}
({\Hc}^1\pi'^*)({\Hc}^{-1}\pi'_*)\M'\to\M',
\end{equation}
since the left-hand is isomorphic to ${\Hc}^{-1}\pi'_*(\M'\boxtimes A_{h,S^*}[1])$ and the morphism is obtained by the connecting morphism of ${\Hc}^{\ssb}\pi'_*$ of the short exact sequence
\begin{equation*}
0\to\delta'_*\M'\to j'_!j'^*(\M'\otimes A_{h,S^*}[1])\to\M'\boxtimes A_{h,S^*}[1]\to 0
\end{equation*}
where $\delta':S^*\to S^*{\times}S^*$ is diagonal and $j'$ its complement.
(We will denote by the same symbol the base change of these morphisms by $X\to S$.) Applying \eqref{4.17} to $j^*\M\otimes E_k^{\lambda}$, we get the natural morphism
\begin{equation} \label{4.18}
({\Hc}^1\pi'^*)({\Hc}^{-1}\pi'_*)(j^*\M\otimes E_k^{\lambda})\otimes\check E_k^{\lambda}\to j^*\M\otimes E_k^{\lambda}\otimes\check E_k^{\lambda}\to j^*\M
\end{equation}
where the last morphism is induced by the natural pairing $E^\lambda_k\otimes\check E^\lambda_k\to A_{h,S^*}$.
The direct sum of \eqref{4.18} for a finite number of $\lambda,k$ induces a surjective morphism onto $j^*\M$ locally on $X$ by an argument similar to Lemma~\ref{L3.2}. Here we can use an inductive argument via the short exact sequences $0\to W_r\M\to W_{r'}\M\to W_{r'}\M/W_r\M\to0$ with $j^*W_{r-1}\M=0$ and also the inductive limit for $k\to\infty$. Note that ${\Hc}^{-1}\pi'_*$ corresponds to taking the kernel of $T{-}id$.
Applying a similar argument to the kernel of the surjective morphism, the assertion can be reduced to the case $\M=\M'\boxtimes j_*\check E^\lambda_k$ [1] for $\M'\in\MHM(X,A)$. The assertion then follows.
\end{proof}

\begin{lem} \label{L4.4}
With the notation and the assumption of {\rm \eqref{4.10}}, we have equivalences of categories
\begin{equation} \label{4.19}
\begin{aligned}
&\psi_p\,\,\,\hbox{{\rm (}or}\,\,\,\varphi_p):\MHM(Y,A)_{\mon,*}\simto\MHM(X,A;T_s,N),\\ &\psi_p\,\,\,\hbox{{\rm (}or}\,\,\,\varphi_p):\MHM(Y,A)_{\mon,!}\simto\MHM(X,A;T_s,N).
\end{aligned}
\end{equation}
\end{lem}

\begin{proof}
We have the canonical isomorphisms
\begin{equation*}
\begin{aligned}
\can&:\psi_p\M\simto\varphi_p\M\qquad\q\hbox{if}\,\,\M\in\MHM(Y,A)_{\mon,!},\\ \Var&:\varphi_p\M\simto\psi_p\M(-1)\q\hbox{if}\,\,\M\in\MHM(Y,A)_{\mon,*}.\\ 
\end{aligned}
\end{equation*}
So it is enough to show the assertion for $\psi_p$.
We prove the case of $\MHM(Y,A)_{\mon,*}$; the other case can be argued similarly.
We construct the inverse functor of $\psi_p$ (locally on $X$) as follows.
Take $\M'\in\MHM(X,A;T_s,N)$ and assume $N^{k+1}\M'=0$.
Using the decomposition by eigenvalues of $T_s$, we may assume that the eigenvalues of $T_s$ on $\M'$ are $\lambda$.
Let $j:Y\setminus X\to Y$ denote the inclusion.
Then, restricted to the subcategory satisfying the conditions
\begin{equation*}
\prod_{\alpha\in\lambda}(T_s-\alpha)\M'=0,\q N^{k+1}\M'=0,
\end{equation*}
the inverse functor $\eta$ is given by applying $j_*$ or $j_!$ to the kernel of the morphism
\begin{equation*}
(T_s\boxtimes T_s^{-1}-id)\oplus(N\boxtimes id-id\boxtimes N):\M'\boxtimes E^\lambda_k\to\M'\boxtimes E_k^\lambda\oplus\M'\boxtimes E_k^\lambda(-1).
\end{equation*}
Indeed, we have a canonical morphism $\psi_p\eta\M'\to\M'$ induced by the projection $\psi_t E_k^\lambda\to A$ as in the proof of Lemma~\ref{L3.2}, and it is an isomorphism by the similar argument.
In particular, $\psi_p$ is faithful.
The same argument applies to the $A$-complexes, and in this case $\psi_p$ and $\eta$ are equivalences of categories by \eqref{2.2}.
Therefore $\eta$ on the mixed Hodge modules is faithful by the faithfulness of the forgetting functor, and $\psi_p$ is fully faithful.
Then the essential surjectivity is shown by using a resolution as in \eqref{4.18}.
This finishes the proof of Lemma~\ref{L4.4}.
\end{proof}

\begin{prop} \label{P4.5}
Let $\pi:Y=X{\times}S\to X$ be a trivial line bundle.
Let $i_P:X{\times}P\into Y,\ j_P:X{\times}(S\setminus P)\into Y$ be natural inclusions for $P\in S^*$ where $\{P\}$ is denoted by $P$ to simplify the notation.
Then for $\M\in\MHM(Y,A)_{\mon}$, we have functorially $\mu(\M)\in\MHM(Y,A)_{\mon,!}$ and $C_P^{-1}(\M),\,C_P^0 (\M)\in\MHM(Y,A)_{\mon}$ together with a commutative diagram:
\begin{equation} \label{4.20}
\begin{gathered}
\xymatrix{\M \ar[d]^u & C_P^{-1}(\M) \ar[l] \ar[r] \ar[d]^v & (\Hc^1\pi^*)(\Hc^{-1}i_P^*)\M \ar[d]^w\\ 
\mu(\M) & C_P^0(\M) \ar[l] \ar[r] & (\Hc^1\pi^*)(\Hc^0\pi_*)(j_P)_!j_P^*\M }
\end{gathered}
\end{equation}
inducing quasi-isomorphisms between the mapping cones of $u,v,w$ which represent $\pi^*\pi_*\M[1]$.
Moreover the underlying triangle of
\begin{equation*}
\M\buildrel u\over\longrightarrow\mu(\M)\longrightarrow C(w)\buildrel +1\over\longrightarrow
\end{equation*}
is canonically isomorphic to
\begin{equation} \label{4.21}
K\longrightarrow\mu(K)\longrightarrow\pi^*\pi_*K[1]\buildrel +1\over\longrightarrow
\end{equation}
obtained by the definition {\rm \eqref{2.1}}.
\end{prop}

\begin{proof}
Consider a cartesian diagram
\begin{equation*}
\xymatrix{Y \ar[d]_{\pi_1}& \Yt \ar[l]_{\tilde\pi_2}\ar[d]^{\tilde\pi_1}\\ X & Y \ar[l]_{\pi_2}}
\end{equation*}
Let $i:Y\to\Yt$ be the diagonal embedding, and $j:\Yt\setminus Y\to\Yt$ its complement We denote by $\tilde i_P,\tilde j_P$ the base change of $i_P, j_P$ by $\pi_2$.
There is a natural short exact sequence
\begin{equation} \label{4.22}
0\to i_*\M\to j_!j^*(\Hc^1\tilde\pi_2^*)\M\to (\Hc^1\tilde\pi_2^*)\M\to 0
\end{equation}
and their direct images by $\tilde\pi_1$ can be defined by the mapping cone of functors:
\begin{equation*}
C((\Hc^{-1}\tilde i_P^*)\to (\Hc^0(\tilde\pi_1)_*)(\tilde j_P)_!\tilde j^*_P).
\end{equation*}
Moreover the direct image of \eqref{4.22} corresponds naturally to \eqref{4.21}, and we define
\begin{equation*}
\mu(\M)=(\Hc^0(\tilde\pi_1)_*)j_!j^*(\Hc^1\tilde\pi_2^*)\M
\end{equation*}
\begin{equation*}
C_P^{-1}(\M)=\M\oplus (\Hc^{-1}\tilde i^*_P)(\Hc^1\tilde\pi_2^*)\M
\end{equation*}
\begin{equation*}
C^0_P(\M)=\Sp_X(\Hc^0(\tilde\pi_1)_*)(\tilde j_P)_!\tilde j^*_P j_!j^*(\Hc^1\tilde\pi^*_2)\M.
\end{equation*}
Here
\begin{equation*}
(\Hc^{-1}\tilde i^*_P)(\Hc^1\tilde\pi^*_2)\M= (\Hc^1\pi^*_2)(\Hc^{-1}i^*_P)\M,
\end{equation*}
and the restriction of $C^0_P(\M)$ to $X{\times}(S\setminus P)$ is isomorphic to that of
\begin{equation*}
(\Hc^0(\tilde\pi_1)_*)(\tilde j_P)_!j^*_Pj_!j^*(\Hc^1\tilde\pi^*_2)\M.
\end{equation*}
We can identify $C(v:C_P^{-1}(\M)\to C^0_P(\M))$ with the direct image of $(\Hc^1\tilde\pi^*_2)\M$ by $(\pi_1)_*$ using the direct image of the complement of the two divisors $\Imm\,\tilde i_P$ and $\Imm\ i$ in $X{\times}(S\setminus P)$, and the horizontal morphisms in \eqref{4.20} are induced by deleting one of the divisors.
So the assertion follows.
\end{proof}

\begin{cor} \label{C4.6}
With the above notation, let $\M\in C^b(\MHM(Y,A)_{\mon})$.
Then we have functorially $\M',\M''\in C^b(\MHM(Y,A)_{\mon})$ together with a morphism $\M'\to\M''$ and a quasi-isomorphism $\M\to [\M'\to\M'']$ such that $\M'^{i}\in\MHM(Y,A)_{\mon,!}$ and $\M''^{i}=\Hc^1\pi^*\Hc^{-1}\pi_*\M''^{i}$.
\end{cor}

\begin{proof}
Let $a:C(v)\to C(u)\to\M[1]$ and $b:C(v)\to C(w)$ be the morphisms induced by \eqref{4.20}.
Let
\begin{equation*}
C=C(a\oplus -b:C(u)[-1]\to\M\oplus C(w)[-1])
\end{equation*}
so that the natural morphism $\M\to C$ is a quasi-isomorphism.
Let $j:Y\setminus X\into Y$ be the natural inclusion, and
\begin{equation*}
\M'=j_!j^*C(C(w)[-1]\to C),\q\M''=C(w).
\end{equation*}
Then we have a canonical morphism $\M'\to\M''$.
Since $j_!j^*$ is an exact functor, and $\M'\to C(C(w)[-1]\to C)$ is a quasi-isomorphism, we get the assertion.
\end{proof}

\begin{lem} \label{L4.7}
With the notation of Lemma~{\rm \ref{L2.1}}, we have natural isomorphisms in the category $D^b\MHM(X,A;T_s,N):$
\begin{equation} \label{4.23}
(\psi_f\M;T_s,N)\simeq (\psi_{f'}\Sp_X\M;T_s,N)
\end{equation}
\begin{equation} \label{4.24}
\begin{aligned}
(\varphi_f\M;T_s,N)&\simeq (\varphi_{f'}\Sp_X\M;T_s,N)\\ \simto(\varphi_{f'}\mu(\Sp_X\M);T_s,N)&\buildrel {\sim}\over\longleftarrow(\psi_{f'}\mu(\Sp_X\M);T_s,N).
\end{aligned}
\end{equation}
for $\M\in D^b\MHM(Y)$ which are compatible with {\rm \eqref{2.8}--\eqref{2.9}}.
Moreover {\rm \eqref{4.23}} and the first isomorphism of {\rm \eqref{4.24}} are compatible with $\can:\psi_1\to\varphi_1$ and $\Var:\varphi_1\to\psi_1(-1)$, where $\psi_1,\varphi_1$ denote the unipotent monodromy part.
\end{lem}

\begin{proof}
It is enough to show \eqref{4.23} and the first isomorphism of \eqref{4.24}.
The assertion is reduced to the commutativity of $\psi_{p}$ with $\psi_{g}$, $\varphi_{g}$ on $q^*\M[1]$ as in the proof of Lemma~\ref{L2.1}.
Let $h:=(g,p):D_XY\to S{\times}S$ so that $h^{-1}(S{\times}\{0\}) =C_XY$ and $h^{-1}(\{0\}{\times}S)=X{\times}S$.
Let $j_1:S^*{\times}S\into S{\times}S$ and $j_2:S{\times}S^*\into S{\times}S$ be the natural inclusions.
We denote their base changes by the same symbols.
Let $\M'$ be the restriction of $q^*\M[1]$ to $h^{-1}(S^*{\times}S^*)$, and $K'$ its underlying $A$-complex.
Then for a local system $L$ on $S^*{\times}S^*$, we can verify the canonical isomorphism
\begin{equation*}
j_{b!}j_{a*}(K'\otimes h^*L)\simto j_{a*}j_{b!}(K'\otimes h^*L)
\end{equation*}
for $(a,b)=(1,2),(2,1)$ using Lemma~\ref{L1.2}.
The assertion follows from Lemma~\ref{L3.2} and the next lemma applied to $j_!j^*\M\to\M\to j_*j^*\M$, since $\psi_{f,\ne 1}=\varphi_{f,\ne 1}$.
\end{proof}

\begin{lem} \label{L4.8}
With the notation of Lemma~{\rm \ref{L3.2}}, we have a canonical isomorphism of $\varphi_f\M$ with the first cohomology of the single complex associated to
\begin{equation*}
\xymatrix{\M \ar[r] & j_*(j^*\M\otimes f^*E_k)\\ j_!j^*\M \ar[r] \ar[u] & j_!(j^*\M\otimes f^*E_k) \ar[u]}
\end{equation*}
for $\M\in\MHM(Y,A)$ if $k\gg 0$ locally on $X$, where the bidegree of $j_!j^*\M$ is $(0,0)$.
This isomorphism is compatible with the defining isomorphism
\begin{equation*}
C(i^*K\to\psi_fK)\simeq\varphi_fK.
\end{equation*}
\end{lem}

\begin{proof}
This is shown for instance in \cite{ext}.
This finishes the proof of Lemma~\ref{L4.8} and Lemma~\ref{L4.7}.
\end{proof}

\section{Thom-Sebastiani theorem for mixed Hodge modules} \label{S5}

Let $X_a,Y_a,X,Y,\Yt$ be as in Theorem~\ref{T1.5}. For $\M_a\in\MHM(X_a;T_s,N)$ with $a=1,2$, we define the {\it twisted exterior product}
\begin{equation*}
\M_1\buildrel T\over\boxtimes\M_2\in\MHM(X;T_s,N),
\end{equation*}
so that it corresponds by \eqref{4.19} to the functor
\begin{equation*}
\pi_*\ssc\boxtimes:\MHM(Y_1,A)_{\mon,!}{\times}\MHM(Y_2,A)_{\mon,!}\to\MHM(Y,A)_{\mon,!},
\end{equation*}
which is induced by the composition (see Proposition~\ref{P5.1} below):
\begin{equation*}
\MHM(Y_1,A)_{\mon,!}{\times}\MHM(Y_2,A)_{\mon,!}\buildrel\boxtimes\over\longrightarrow\MHM(\Yt,A)\buildrel{\pi_*}\over\longrightarrow D^b\MHM(Y,A).
\end{equation*}
We have a more explicit definition of $\buildrel T\over\boxtimes$ as follows.
Let
\begin{equation*}
\M_a=(M_a,F,K_a,W;T_s,N)\q\hbox{with}\,\,\,\gamma_a:\DR(M_a)\simto K_a\otimes\C.
\end{equation*}
For $-1<\alpha\les 0$, let $M^\alpha_a=\Ker(T_s-\exp(-2\pi i\alpha))\subset M^\alpha_a$ (same for $K_a\otimes\C)^\alpha )$.
We define
\begin{equation*}
\M_1\buildrel T\over\boxtimes\M_2=(M,F,K,W;T_s,N)
\end{equation*}
by $M=M_1\boxtimes M_2,\ K=K_1\boxtimes K_2,T_s=T_s\boxtimes T_s,\ N=N\boxtimes id + id\boxtimes N$, and
\begin{equation} \label{5.1}
F_p(M_1^\alpha\boxtimes M_2^\beta)=\begin{cases}\sum_{i+j=p+1}(F_iM_1^\alpha\boxtimes F_jM_2^\beta)&\h{if}\,\,\,\alpha +\beta\les -1,\\ \sum_{i+j=p}(F_iM_1^\alpha\boxtimes F_jM_2^\beta)&\h{if}\,\,\,\alpha +\beta> -1,\\ \end{cases}
\end{equation}
\begin{equation} \label{5.2}
W_k(M_1^\alpha\boxtimes M_2^\beta)=\begin{cases}\sum_{i+j=k}(W_iM_1^\alpha\boxtimes W_jM_2^\beta) &\h{if}\,\,\,\alpha\beta=0,\\ \sum_{i+j=k-1}(W_iM_1^\alpha\boxtimes W_jM_2^\beta) &\h{if}\,\,\,\alpha\beta\ne 0,\,\alpha +\beta\ne -1,\\ \sum_{i+j=k-2}(W_iM_1^\alpha\boxtimes W_jM_2^\beta) &\h{if}\,\,\,\alpha +\beta=-1,\\ \end{cases}
\end{equation}
(same for $W_k((K_1\otimes\C)^\alpha\boxtimes (K_2\otimes\C)^\beta))$ where $\alpha,\beta\in(-1,0]$.
Here $W$ on $K$ is defined over $A$ by using the Galois group.
The isomorphism
\begin{equation*}
\gamma:\DR(M_1\boxtimes M_2)\simto (K_1\otimes\C)\boxtimes (K_2\otimes\C)
\end{equation*}
is given by the composition of $\gamma_1\boxtimes\gamma_2$ with an isomorphism
\begin{equation*}
\Bt:(K_1\otimes\C)\boxtimes (K_2\otimes\C)\simto(K_1\otimes\C)\boxtimes (K_2\otimes\C),
\end{equation*}
defined by
\begin{equation} \label{5.3}
\Bt(u\boxtimes v)=\begin{cases}B_{\alpha-N,\beta-N}(u\boxtimes v)\q\q\q&\h{if}\,\,\,\alpha+\beta\les -1,\\ (\alpha{+}\beta{+}1{-}N)B_{\alpha-N,\beta-N}(u\boxtimes v)&\h{otherwise},\\ \end{cases}
\end{equation}
where we define for $u\boxtimes v\in (K_1\otimes\C)^\alpha\boxtimes(K_2\otimes\C)^\beta$
\begin{equation*}
B_{\alpha-N,\beta-N}(u\boxtimes v):=
\msum_{i,j\ges 0}\,\hbox{$\frac{1}{i!j!}$}\,B^{i,j}(\alpha,\beta)((-N)^iu\boxtimes (-N)^jv),
\end{equation*}
with
\begin{equation*}
B^{i,j}(\alpha,\beta):=\rd_{\alpha}^i\rd_{\beta}^jB(\alpha,\beta),\q B(\alpha,\beta):=\int_0^1x^{\alpha}(1{-}x)^{\beta}dx.
\end{equation*}
Note that the latter is the beta function up to the shift of $\alpha$, $\beta$ by 1.
In \eqref{5.3}, $N$ acts on $(K_1\otimes\C)\boxtimes(K_2\otimes\C)$ by $N\boxtimes\id+\id\boxtimes N$.
We define $(M_1;T_s,N)\buildrel T\over\boxtimes(M_2;T_s,N)$ for analytic spaces $X_1,X_2$ and for $(M_a;T_s,N)\in\MHM (X_a;T_s,N)$ using local embeddings of analytic spaces into complex manifolds, see \cite{mhp}, \cite{mhm}.
\ms
The following proposition shows that the above two definitions of $\buildrel T\over\boxtimes$ coincide.

\begin{prop} \label{P5.1}
With the notation of Proposition~{\rm \ref{P2.2}}, let $\M_a\in\MHM(Y_a,A)_{\mon,!}$.
Then
\begin{equation*}
\pi_*(\M_1\boxtimes\M_2)\in\MHM(Y,A)_{\mon,!}\,\,\,\hbox{{\rm (}resp.}\,\,\,D^b(\MHM(Y,A)_{\mon,!})),
\end{equation*}
and we have, with the second definition of $\buildrel T\over\boxtimes$, a canonical isomorphism
\begin{equation} \label{5.4}
(\psi_p\pi_*(\M_1\boxtimes\M_2);T_s,N)\simeq (\psi_{p_1}\M_1;T_s,N)\buildrel T\over\boxtimes (\psi_{p_2}\M_2;T_s,N)
\end{equation}
in $\MHM(X,A;T_s,N)$ which is compatible with {\rm \eqref{2.10}}.
Moreover, it naturally extends to an isomorphism in $D^b\MHM(X,A;T_s,N)$ for $\M_a\in D^b(\MHM (Y_a,A)_{\mon,!})$.
\end{prop}

\begin{proof}
The functor real in \cite{BBD} commutes with $\psi_p,\pi_*$, etc.
and it is enough to show the assertion in the case $\M_a\in\MHM(Y_a,A)_{\mon,!}$.
Then we may assume $X_1,X_2$ smooth.
Put $\M=\M_1\boxtimes\M_2$ and $\M'=\pi_*\M$.
The first assertion $\M'\in\MHM(Y,A)_{\mon,!}$ follows from Proposition~\ref{P2.2}.
Let $M$, $M'$, $M_a$ be the underlying $\D$-modules of $\M$, $\M'$, $\M_a$.
In the notation of Section~\ref{S4}, $(\pi_*M)^{\alpha}$ will be denoted by $\Mo^{\alpha}$, since $\pi$ in this paragraph denotes the projection of $\Yt$ to $Y$.
Then we have
\begin{equation*}
[M_a]=\mopl_{\alpha}\,\Mo_a^{\alpha},\q[M]=\mopl_{\alpha}\,\Mo^{\alpha},\q[M']=\mopl_{\alpha}\,\Mo'^{\alpha},
\end{equation*}
and
\begin{equation*}
[M]=[M_1]\boxtimes[M_2],\q\Mo^{\alpha}=\motim_{\alpha_1+\alpha_2=\alpha}\,\Mo^{\alpha_1}_1\boxtimes\Mo^{\alpha_2}_2.
\end{equation*}
By definition, $\pi_*M=\pi_*\DR_{\Yt/Y}M$, where
\begin{equation*}
\DR_{\Yt/Y}M=C(d:M\to\Omega_{\Yt/Y}^1\otimes M),
\end{equation*}
and $\Omega_{\Yt/Y}^1$ is trivialized by $dt_1=-dt_2$ (mod $\pi^*\Omega_Y^1)$.
So we have
\begin{equation} \label{5.5}
[M']=\Coker(\rd_{t_1}{-}\,\rd_{t_2}: [M_1]\boxtimes[M_2]\to[M_1]\boxtimes[M_2]),
\end{equation}
where $\Mo^{\alpha_1}_1\boxtimes\Mo^{\alpha_2}_2$ in the cokernel of $\rd_{t_1}{-}\,\rd_{t_2}$ is sent to a subsheaf of $\Mo'^{\alpha_1+\alpha_2+1}$ by the isomorphism.
Since $\rd_{t_a}:\Mo^{\alpha}_a\simto\Mo^{\alpha-1}_a$, we get a canonical isomorphism for $\alpha\in (-1,0]$:
\begin{equation} \label{5.6}
\Mo'^{\alpha}=\Biggl(\,\bigoplus_{\substack{\scriptstyle\alpha_1,\alpha_2\in (-1,0]\\ \scriptstyle\alpha_1+\alpha_2+1=\alpha}}\Mo^{\alpha_1}_1\boxtimes\Mo^{\alpha_2}_2\Biggr)\bigoplus\Biggl(\,\bigoplus_{\substack{\scriptstyle\alpha_1,\alpha_2\in (-1,0]\\ \scriptstyle\alpha_1+\alpha_2=\alpha}}\Mo^{\alpha_1}_1\boxtimes\Mo^{\alpha_2}_2\Biggr).
\end{equation}
Here we use in case $\alpha_1+\alpha_2=\alpha$ the following isomorphism
\begin{equation*}
\rd_t:\Mo'^{\alpha+1}\simto\Mo'^{\alpha},
\end{equation*}
where $t$ is the affine coordinate of $Y$ over $X$ such that $\pi^*t=t_1{+}t_2$.
So we get the isomorphism \eqref{5.4} for the underlying $\D$-modules.
With the notation of Section~\ref{S4}, set
\begin{equation*}
{}^{\theta}\!M_a:=\Ker(t_a\rd_{t_a}:M_a|_{Y^*_a}\to M_a|_{Y^*_a}).
\end{equation*}
Then ${}^{\theta}\!M_a$ is monodromical, i.e.
the restriction of ${}^{\theta}\!M_a$ to $\{x\}{\times}S^*$ is locally constant.
We have a canonical isomorphism
\begin{equation} \label{5.7}
\lambda:\mopl_{\alpha\in (-1,0]}\,\Mo^{\alpha}_a\simto {}^{\theta}\!M_a|_{X_a{\times}\{1\}}=\psi_{p_a}({}^{\theta}\!M_a)
\end{equation}
such that the monodromy of ${}^{\theta}\!M_a$ corresponds to $\exp(2\pi i(-\alpha+N))$ with $N=-(t_a\rd_{t_a}-\alpha)$ on $\Mo^{\alpha}_a$, where the last isomorphism is shown as in \eqref{2.2}.
Indeed, for $\alpha\in (-1,0]$, we have a natural morphism
\begin{equation} \label{5.8}
{}^{\theta}\!M_a|_{X_a{\times}\{1\}}=\psi_{p_a}({}^{\theta}\!M_a)\to(\pi_a)_*(i_a)_*(\Oc_{Y_a^*}\otimes_{\pi_a^{-1}\Oc_{X_a}}{}^{\theta}\!M_a)
\end{equation}
which is defined by
\begin{equation*}
{}^{\theta}\!M_a|_{X_a{\times}\{1\}}\ni u\mapsto t_a^{\alpha-N}(u),
\end{equation*}
if $u$ is an eigenvector of the semisimple part of the monodromy of ${}^{\theta}\!M_a$ with eigenvalue $\exp(-2\pi i\alpha)$.
Moreover, \eqref{5.8} is naturally factorized by
\begin{equation*}
\Mo^{\alpha}_a\to(\pi_a)_*M_a\to(\pi_a)_*(i_a)_*(\Oc_{Y_a^*}\otimes_{\pi_a^{-1}\Oc_{X_a}}{}^{\theta}\!M_a),
\end{equation*}
to induce the inverse of \eqref{5.7}, where
\begin{equation*}
t_a^{\alpha-N}u:=
\msum_{k\ges 0}\,t_a^{\alpha}(\log t_a)^k(-N)^k u/k!.
\end{equation*}
Taking $\DR_{X_a}$, the isomorphism \eqref{5.7} induces the isomorphism
\begin{equation*}
\DR_{X_a}\psi_{p_a}\simeq\psim_{p_a}\DR_{Y_a}.
\end{equation*}
By \eqref{5.7}--\eqref{5.8}, we can compare the isomorphisms \eqref{2.10} and \eqref{5.6}, and get the twist of the $A$-structure \eqref{5.3}, since the trace morphism $Tr:\pi''_!A_{X{\times}L}\simto A_X[-1]$ in Lemma~\ref{L2.3} is expressed by the integration of one-forms along $L$.
Indeed, we restrict to $\pi^{-1}(X{\times}\{1\})$ the direct image $\pi_*M=\pi_*\DR_{\Yt/Y}M$ which is essentially expressed by \eqref{5.5}, and then take $\DR_X$.
It has a natural morphism to $\pi_*\DR_{\Yt}M|_{X{\times}\{1\}}$, and its composition with the trace morphism is given by the integration of one-forms as above.
The beta function appears from the integration over $L=(0,1)$ contained in $\{t_1{+}t_2=1\}$ by using the above formula for $t_a^{\alpha-N}u$ together with
\begin{equation*}
\rd_{\alpha}x^{\alpha}=(\log x)\,x^{\alpha},\q\rd_{\beta}(1{-}x)^{\beta}=\log (1{-}x)\,(1{-}x)^{\beta},
\end{equation*}
where $x=t_1=1-t_2$.
In case $\alpha_1+\alpha_2>-1$, the action of $\alpha_1+\alpha_2+1-N$ on $B_{\alpha-N,\beta-N}(u\boxtimes v)$ comes from the isomorphism after \eqref{5.6}, since $[dt_1\otimes(m_1\boxtimes m_2)]\in M'$ for $m_a\in\Mo_a^{\alpha_a}$ is expressed as $t^{\alpha_1+\alpha_2+1-N}u'$ with $u'$ a locally constant multivalued section, and we have
\begin{equation*}
\rd_t(t^{\alpha_1+\alpha_2+1-N}u')= (\alpha_1{+}\alpha_2{+}1{-}N)(t^{\alpha_1+\alpha_2-N}u').
\end{equation*}
For the shift of the Hodge filtration \eqref{5.1}, it is enough to show that the isomorphism \eqref{5.6} is compatible with the Hodge filtration $F$ up to the shift as in \eqref{5.1}.
We first show the inclusion
\begin{equation} \label{5.9}
F_p\Mo'^{\alpha}\supset\Biggl(\,\bigoplus_{\substack{\scriptstyle\alpha_1,\alpha_2\in (-1,0],\\ \scriptstyle\alpha_1+\alpha_2+1=\alpha}}F_{p+1}(\Mo^{\alpha_1}_1\boxtimes\Mo^{\alpha_2}_2)\Biggr)\bigoplus\Biggl(\,\bigoplus_{\substack{\scriptstyle\alpha_1,\alpha_2\in (-1,0],\\ \scriptstyle\alpha_1+\alpha_2=\alpha}}F_p(\Mo^{\alpha_1}_1\boxtimes\Mo^{\alpha_2}_2)\Biggr).
\end{equation}
Since $\alpha>-1$, it is enough to show the inclusion at $X{\times}\{t=1\}$.
By definition, the Hodge filtration $F$ on $M'$ is obtained by using the natural compactification of $\pi:\Yt\to Y$, where $(M,F)$ is extended to the compactification using the filtration $V$ along the divisor at infinity (see \cite[3.2.3.2]{mhp}), where $V_0$ corresponds to $V^{-1}$ in this note).
The restriction of
\begin{equation*}
dt_1\otimes(m_1\boxtimes m_2)\in\Omega_{\Yt/Y}^1\otimes F_p(\Mo^{\alpha_1}_1\boxtimes\Mo^{\alpha_2}_2)\q\hbox{to}\q X{\times}\{t_1{+}t_2=1\}
\end{equation*}
is naturally extended over the compactification, and belongs to $V^{-\alpha_1-\alpha_2-2}$ since $\Omega_{\PP^1}^1\simeq\Oc_{\PP^1}(-2)$.
By [loc. cit.], we get
\begin{equation*}
\begin{array}{cccccccccccc}[dt_1\otimes(m_1\boxtimes m_2)]\in F_{p-1}M'\hfill&{\rm if}\q\alpha_1+\alpha_2\les -1\hfill\\ \rd_t[dt_1\otimes(m_1\boxtimes m_2)]\in F_{p}M'\hfill& {\rm otherwise.}\raise10pt\hbox{}\hfill\\ \end{array}
\end{equation*}
Indeed, the filtration $F$ on $\DR_{\Yt/Y}M$ is defined by
\begin{equation*}
F_p\DR_{\Yt/Y}M=C(d:F_pM\to\Omega_{\Yt/Y}^1\otimes F_{p+1}M),
\end{equation*}
and we have
\begin{equation*}
\rd_t[dt_1\otimes(m_1\boxtimes m_2)] =[dt_1\otimes(\rd_{t_1}m_1\boxtimes m_2)] =[dt_1\otimes(m_1\boxtimes\rd_{t_2}m_2)],
\end{equation*}
where $\rd_{t_1}$ has zero of order $2$ at infinity.
So we get \eqref{5.9}.
Since the functor $(M_1,M_2)\to M'$ is exact for both factors, we may assume $M_1,M_2$ pure by \eqref{5.9}.
Then the assertion is reduced to the case of variations of Hodge structures, since the Hodge filtration of a pure Hodge module (with strict support) is determined to its restriction to a dense Zariski-open subset of the support, see \cite[3.2.2]{mhp}.
In this case, we may assume also $X_1=X_2=\pt$ using the non-characteristic pull-backs by the base change by $\{x_1\}{\times}\{x_2\}\to X_1{\times}X_2$, and its commutativity with the direct images.
Since we are considering only the underlying filtered $\D$-modules, we may assume that $M_a|{S^*}$ is a line bundle with integrable connection by \eqref{4.14}.
Here the Hodge filtration on $M_a|{S^*}$ is trivial, i.e.
$\Gr_p^F M_a|{S^*}=0$ for $p\ne p_a$.
Let $\exp(-2\pi i\alpha_a)$ be the monodromy of $M_a|{S^*}$ with $\alpha_a\in (-1.0]$.
By \eqref{5.6} it is enough to compare the dimension of the Hodge filtration on both sides of \eqref{5.6}.
For $\alpha\in (-1,0]$ such that $\alpha_1+\alpha_2-\alpha\in\Z$, we have by Proposition~\ref{P4.3}
\begin{equation*}
\dim F_p\Mo'^{\alpha}=\dim F_p\Mo'=\dim F_p\Mo'|_{\{t=1\}}
\end{equation*}
Let $S'=S\setminus\{0,1\}$, and $L'$ the restriction of $M$ to $\{t_1{+}t_2=1\}\setminus\{t_1t_2=0\}\simeq S'$.
Then $L'$ is a line bundle with integrable connection, and its monodromies of around $0,1$ are $\exp(-2\pi i\alpha_1)$, $\exp(-2\pi i\alpha_2)$.
Let $\bar S=\PP^1$ with $j_{\infty}:S\to\bar S$.
Let $L$ be Deligne's extension of $L'$ to $\bar S$ such that the eigenvalues of the residue of the connection at $0,1$ (resp. $\infty$) are contained in $(0,1]$ (resp. $[0,1)$).
It is obtained by restricting $j_{\infty *}(M|_{S{\times}\{t=1\}})$ to $V^{>0}$ (resp. $V^0$) at $0,1$ (resp. $\infty$).
By \cite[3.11.2]{mhm}, the direct image $(M',F)|_{\{t=1\}}$ is calculated by the hypercohomology of the filtered complex
\begin{equation*}
K=C(L\to\Omega_{\bar S}^1(\log\Sigma)\otimes L),\q F_pK=C(F_p L\to\Omega_{\bar S}^1(\log\Sigma)\otimes F_{p+1}L),
\end{equation*}
where $\log\Sigma=\{0,1,\infty\}$ and $F_p L=L$ if $p\ges p':= p_1+p_2$ and zero otherwise, i.e.
the filtration on $K$ is the filtration $\sigma_{\ges p}$ (see \cite{D2}) up to a shift.
We have
\begin{equation*}
\begin{array}{cccccccccccc}L=\Oc_{\PP^1}(-1),\hfill&\Omega_{\bar S}^1(\log\Sigma)\otimes L=\Oc_{\PP^1}(0)\hfill&\hbox{\rm if}\q\alpha_1+\alpha_2\les -1,\hfill\\ L=\Oc_{\PP^1}(-2),\hfill&\Omega_{\bar S}^1(\log\Sigma)\otimes L =\Oc_{\PP^1}(-1)\hfill&{\rm otherwise.}\raise10pt\hbox{}\hfill\\ \end{array}
\end{equation*}
So we get the assertion on the Hodge filtration, see \eqref{5.1}.
Now we show the assertion about the weight filtration, see \eqref{5.2}.
Let
\begin{equation*}
\begin{aligned}
&Y'=X{\times}\{t_1{+}t_2=1\}\,(\simeq X{\times}S),\\ &Y''=Y'\setminus\{t_1t_2=0\},\q\Yo'=X{\times}\PP^1,
\end{aligned}
\end{equation*}
with natural morphisms
\begin{equation*}
\begin{aligned}
&j':Y''\to Y',\q j_{\infty}:Y'\to\Yo',\\ &\pi':Y'\to X{\times}\{t=1\},\q\bar\pi':\Yo'\to X{\times}\{t=1\},
\end{aligned}
\end{equation*}
so that $\bar\pi'j_{\infty}=\pi'$.
Let $\M''$ denote the restriction of $\M$ to $Y''$ (i.e.
$\Hc^{-1}i^*\M|_{Y''}$ with $i:Y'\into\Yt$ the inclusion) so that the restriction of $\M$ to $Y'$ is $j'_!\M''$.
Then
\begin{equation*}
\M'|_{\{t=1\}}=\Hc^0\pi'_*j'_!\M'',
\end{equation*}
and
\begin{equation} \label{5.10}
\Hc^i\pi'_*j'_!\M''=0\q{\rm for}\q i\ne 0.
\end{equation}
So we have a filtration $\Wt$ of $\Hc^0\pi'_*j'_!\M''$ such that
\begin{equation*}
\Wt_k\Hc^0\pi'_*j'_!\M''=\Hc^0\pi'_*j'_!(W_{k+1}\M''),
\end{equation*}
since \eqref{5.10} holds also for $W_k\M''$ and $\Gr_k^W\M''$.
Here the shift of filtration by one comes from the smooth pull-back of relative dimension one.
Let $(M'')^{\alpha_1,\alpha_2}\subset M''$ denote the restriction to $Y''$ of the sub-module of $M$ generated by $M_1^{\alpha_1}\boxtimes M_2^{\alpha_2}$, so that we have a decomposition $M''=\bigoplus_{\alpha_1,\alpha_2\in (-1,0]}(M'')^{\alpha_1,\alpha_2}$.
It is enough to show
\begin{equation*}
\begin{aligned}
&\Gr_k^W\Gr_{k'}^{\Wt}\Hc^0\pi'_*j'_!(M'')^{\alpha_1,\alpha_2}=0\\ &\q\q{\rm for}\,\,\begin{cases}k'\ne k\q\q&\h{if}\q\alpha_1\alpha_2=0,\\ k'\ne k{-}1&\h{if}\q\alpha_1\alpha_2\ne 0,\,\,\alpha_1{+}\alpha_2\ne -1,\\ k'\ne k{-}2&\h{if}\q\alpha_1{+}\alpha_2=-1.\\ \end{cases}
\end{aligned}
\end{equation*}
So we may assume that $M''$ is pure of weight $k'{+}1$, and moreover, it is a line bundle and $X=\pt$ by the same argument as above.
If the monodromy around $0$ or $1$ (resp. $\infty$) is trivial, we have
\begin{equation*}
\Gr_{k'}^W j_{\infty *}j'_!M''\ne 0\q\hbox{(resp.}\,\,\Gr_{k'+2}^Wj_{\infty *}j'_!M''\ne 0),
\end{equation*}
and it is zero otherwise.
Then we can verify the assertion by using the weight spectral sequence.
The detail is left to the reader.
This finishes the proof of Proposition~\ref{P5.1}.
\end{proof}

\begin{prop} \label{P5.2}
In the notation of Proposition~{\rm \ref{P2.4}}, let $\M_a\in D^b(\MHM(Y_a,A)_{\mon})$.
We have $\pi_*(\M_1\boxtimes\M_2)$ in $D^b(\MHM(Y,A)_{\mon})$ compatible with the definition in Proposition~{\rm \ref{P3.4}} by the natural functor $D^b(\MHM(Y,A)_{\mon})\to D^b(\MHM(Y,A))$, and there is an canonical isomorphisms in $D^b(\MHM(X,A;T_s,N)):$
\begin{equation*}
\varphi_p\pi_*(\M_1\boxtimes\M_2)\simeq\varphi_{p_1}\M_1\buildrel T\over\boxtimes\varphi_{p_2}\M_2,
\end{equation*}
\end{prop}

\begin{proof}
By Corollary~\ref{C4.6}, we may assume $\M_a=[\M'_a\to\M''_a]\,(a=1,2)$ such that
\begin{equation*}
\M'^{i}_a\in\MHM(Y_a,A)_{\mon,!},\q\M''^{i}_a=\Hc^1\pi^*\Hc^{-1}\pi_*\M''^{i}_a.
\end{equation*}
The assertion then follows from Proposition~\ref{P5.1}.
\end{proof}

\begin{thm} \label{T5.3}
With the notation of Theorem~{\rm \ref{T2.5}}, let $\M_a\in D^b (\MHM(X_a,A))$.
We have a canonical isomorphism in $D^b(\MHM(X;T_s,N)):$
\begin{equation*}
\varphi_f(\M_1\boxtimes\M_2)\simeq\varphi_{f_1}\M_1\buildrel T\over\boxtimes\varphi_{f_2}\M_2.
\end{equation*}
Here we assume that $\varphi_{f_a-c_a}\M_a=0$ for any $c_a\in\C^*$ replacing $X_a$ with a sufficiently small open neighborhood of $f_a^{-1}(0)\subset X_a$ for $a=1$ or $2$.
\end{thm}

\begin{proof}
This follows from Theorem~\ref{T3.6}, Lemma~\ref{L4.7} and Proposition~\ref{P5.2}.
Note that $i_*i^*$ can be defined as in the proof of Theorem~\ref{T3.6} with $g\eq f_1{-}f_2$, since $s_2=t_1{-}t_2$.
\end{proof}

\end{document}